\documentclass[11pt,a4paper,reqno]{amsart}
\usepackage{a4wide}
\usepackage{amsfonts,amsmath,amssymb,amsthm,enumerate,color,hyperref}
\usepackage[latin1]{inputenc}
\usepackage[dvips,final]{graphicx}
\usepackage{natbib}

\newcommand{\eqdef}{=}
\newcommand{\eqsp}{\;}

  {\end{enumerate}}

\def\PE{\mathbb{E}}
\newcommand\CPE[2]{\PE\left[ \left. #1 \right| #2 \right]}
\def\PP{\mathbb{P}}
\newcommand\CPP[2]{\PP\left[ \left. #1 \right| #2 \right]}

\def\Var{\mathop{\rm Var}\nolimits}

\def\rmi{\textrm{i}}
\def\rme{\textrm{e}}
\def\trace{\textrm{Tr}}
\def\rem{R}
\def\fmin{f_{\min}}
\def\fmax{f_{\max}}
\newtheorem{theo}{Theorem}
\newtheorem{prop}{Proposition}
\newtheorem{lemma}{Lemma}

\theoremstyle{remark}
\newtheorem{remark}{Remark}

\def\rset{\mathbb{R}}
\def\zset{\mathbb{Z}}

\begin{document}

\title[Frequency estimation based on the cumulated Lomb-Scargle periodogram]
{Frequency estimation based on the cumulated Lomb-Scargle periodogram}
\author{C. L\'evy-Leduc}
\author{E. Moulines}
\author{F. Roueff}
\address{GET/Télécom Paris, CNRS LTCI, 46, rue Barrault, 75634 Paris Cédex 13, France.}
\email{[levyledu,moulines,roueff]@tsi.enst.fr}
\keywords
{Period estimation; Frequency estimation; Irregular sampling; Semiparametric estimation; Cumulated Lomb-Scargle
periodogram.}
\date{\today}

% \thispagestyle{empty}

% \centerline{\textbf{FREQUENCY ESTIMATION BASED ON THE CUMULATED}}
% \centerline{\textbf{LOMB-SCARGLE PERIODOGRAM}}\ \\ \ \\

% \centerline{By C. L\'evy-Leduc, E. Moulines and F. Roueff}\ \\ \ \\

% \underline{Corresponding author} : C. L\'evy-Leduc\ \\ \ \\

% Address: \ \\ 

% Ecole Nationale Sup\'erieure des T\'el\'ecommunications

% 37/39 Rue Dareau

% 75014 Paris (FRANCE)

% \newpage
% \thispagestyle{empty}
% \null
% \newpage

% \addtocounter{page}{-2}

\maketitle

\vspace{-5mm}
\centerline{\textit{GET/Télécom Paris, CNRS LTCI}}

\begin{abstract}
We consider the problem of estimating the period of an unknown periodic function observed in additive noise sampled at
irregularly spaced
time instants in a semiparametric setting. To solve this problem, we propose a
novel estimator based on the cumulated Lomb-Scargle periodogram. We prove that this estimator is
consistent, asymptotically Gaussian and we provide an explicit expression
of the asymptotic variance.
Some Monte-Carlo experiments are then presented to support our claims.
\end{abstract}

\section{Introduction}
The problem of estimating the frequency of a periodic function corrupted by additive noise is ubiquitous and has
attracted a lot of research efforts in the last three decades. Up to now, most of these contributions have been devoted to  regularly sampled
observations; see e.g. \cite{quinn:hannan:2001} and the references therein.
In  many applications however, the observations are sampled at irregularly spaced time instants: examples occur in different fields,
including among others  biological rhythm research from free-living animals (\cite{ruf:1999}), unevenly spaced gene
expression time-series analysis (\cite{glynn:2006}),
or the analysis of brightness of periodic stars (\cite{hall:reimann:rice:2000,thiebaut:roques:2005}). In the latter case, for example, irregular observations come
from missing observations due to poor weather conditions (a star can be observed on most nights but not all nights),
and because of the variability of the observation times.
In the sequel, we consider the following model:
\begin{equation}\label{eq:Y}
Y_j=s_{\star}(X_j)+\varepsilon_j ,\quad j=1,2,\dots,n \;,
\end{equation}
where $s_\star$ is an unknown (real-valued) $T$--periodic function on the real line, $\{X_k\}$ are the sampling
instants and $\{ \varepsilon_k\}$ is an
additive noise. Our goal is to construct a consistent, rate optimal and easily
computable estimator of the frequency $f_0=1/T$ based on the
observations $\{ (X_i,Y_i) \}_{i=1,\dots,n}$ in a semiparametric
setting, where $s_{\star}$ belongs to some function space.
To our best knowledge, the only attempt to rigorously derive such semiparametric estimator is due to \citet*{hall:reimann:rice:2000}, who propose to use the
least-squares criterion defined by $S_n(f)=\sum_{k=1}^n \left(Y_k-\hat{s}(X_k|f)\right)^2$ where $\hat{s}(x|f)$ is a
nonparametric kernel estimator of $s_{\star}(x)$, adapted to a given frequency $f$, from the observations
$(X_j,Y_j),\;j=1,\dots,n$. For an appropriate choice of $\hat{s}(x|f)$,
the minimizer of $S_n(f)$ has been shown to
converge at the parametric rate and to achieve the optimal asymptotic variance, see~\cite{hall:reimann:rice:2000}.
%However, the consistency is obtained under the assumption that a preliminary estimator of the period is known, whose
%rate of convergence is already close to be optimal (see \cite{hall:reimann:rice:2000}[condition (c), p. 554]).

Here, we propose to estimate the frequency by maximizing the Cumulated Lomb-Scargle Periodogram (CLSP), defined as
\begin{equation}\label{eq:CumPEr}
\Lambda_n(f)\eqdef\frac{1}{n^2}\sum_{k=1}^{K_n}\left|\sum_{j=1}^n Y_j \ \textrm{e}^{-2\textrm{i}k\pi f X_j}\right|^2 \; ,
\end{equation}
where $K_n$ denotes the number of cumulated harmonics, assumed to be slowly increasing with $n$.
Considering such an estimator is very natural since this procedure might be seen as an adaptation of the algorithm
proposed by \citet*{quinn:thomson:1991} obtained
by replacing the periodogram by the Lomb-Scargle periodogram, introduced in \cite{lomb:1976} (see also
\cite{scargle:1982}) to account for irregular sampling time instants. Note also that such an estimator can
be easily implemented and efficiently computed using \cite[p.~581]{numericalrecipes}.
We will show that the estimator based on the maximization of the
cumulated Lomb-Scargle periodogram $\Lambda_n(f)$
is consistent, rate optimal and asymptotically Gaussian.
It is known that frequency estimators based on the cumulated periodogram are
optimal in terms of rate and asymptotic variance in the cases of continuous time observations
and regular sampling (see \cite{golubev:1988,quinn:thomson:1991,gassiat:levyleduc:2006}).
We will see that,
somewhat surprisingly, at least under renewal assumptions on the observation times,
the asymptotic variance is no longer optimal in the irregular sampling case investigated here.
However, because of its numerical simplicity, we believe that the CLSP estimator is a
sensible  estimator, which may be used as a starting value of more sophisticated and computationally intensive
techniques, see e.g. ~\citet*{hall:reimann:rice:2000}. The numerical experiments that we have conducted clearly support these
findings.

The paper is organized as follows. In Section~\ref{results}, we state our main results (consistency and asymptotic normality) and provides sketches of the proofs.
In Section~\ref{implementation},
we present some numerical experiments to compare
the performances of our estimator with the estimator of \citet*{hall:reimann:rice:2000}.
In Section~\ref{sec:proofs}, we provide some auxiliary results and we detail the steps that are omitted in the proof
sketches of the main results.

\section{Main results}\label{results}
%\subsection{Identifiability}
%We consider the model (\ref{eq:Y}) with assumptions (i)--(iii) and the following additional assumption:
%\begin{enumerate}[(iv)]
%\item the distribution of $V_1$ has a non-zero absolutely continuous part with respect to the Lebesgue measure.
%\end{enumerate}
%\begin{theo}\label{thm:ident}
%Denote by $\mathcal{T}$ the space of locally integrable periodic functions.
%Under (i)--(iv), the model~(\ref{eq:Y}) parameterized by $s_{\star}\in \mathcal{T}$ is identifiable.
%\end{theo}
Define the Fourier coefficients of a locally integrable $T$-periodic function $s$ by
\begin{equation}\label{fourierCoeff}
c_k(s) \eqdef \frac{1}{T} \int_0^{T} s(t) \, \rme^{-2\rmi k\pi t/T} dt,\quad k\in\zset
\quad\text{so that}\quad
s(t)=\sum_{p\in\zset} c_p(s)\rme^{2\rmi\pi p t/T} \; ,
\end{equation}
when this expansion is well defined.
Recall that the frequency $f_0=1/T$ of $s_{\star}$ is
here the parameter of interest. Consider
the least-squares criterion based on observations $\{(X_i,Y_i) \}_{i=1,\dots,n}$,
\begin{eqnarray}\label{logvrais}
L_n(f,  \mathbf{c} )\eqdef\sum_{j=1}^n \left(Y_j - \sum_{k=-K_n}^{K_n} c_k \ \textrm{e}^{2\textrm{i}k\pi f X_j}  \right)^2 \eqsp , \quad {\mathbf{c}} \eqdef [c_{-K_n}, \dots, c_{K_n}]^T
\end{eqnarray}
where $\{ K_n \}$ is the number of harmonics. For a given frequency $f$, the coefficients $\tilde{\mathbf{c}}_n(f)= [\tilde{c}_{-K_n}, \dots, \tilde{c}_{K_n}]$ which minimize~(\ref{logvrais})
solve the system of equations $G_n(f) \tilde{\mathbf{c}}_n(f)=  n \hat{\mathbf{c}}_n(f)$, where
the (Gram) matrix $G_n(f) \eqdef [ G_{n,k,l}(f) ]_{-K_n \leq k,l \leq K_n}$ and the vector
$\hat{\mathbf{c}}_n(f)= [\hat{c}_{-K_n}(f), \dots, \hat{c}_{K_n}(f)]$ are defined by:
\begin{equation}
\label{eq:cchap}
G_{n,k,l}(f) \eqdef \sum_{j=1}^n \rme^{-2 \rmi (k-l) \pi f X_j} \quad \text{and} \quad \hat{c}_l(f)\eqdef\frac{1}{n}\sum_{j=1}^n Y_j \textrm{e}^{-2\textrm{i}l\pi f X_j} \;.
\end{equation}
An estimator for the frequency $f_0$  can then be obtained by minimizing the residual sum of squares
\begin{equation}
\label{eq:ResidualSumofSquares}
f \mapsto L_n(f,\{\tilde{c}_k(f)\}) \eqdef \sum_{j=1}^n Y_j^2 - n^2\mathbf{\hat{c}}^T_n(f) G^{-1}_n(f) \mathbf{\hat{c}}_n(f) \eqsp.
\end{equation}
Note that computing $\tilde{\mathbf{c}}_n(f)$ is numerically cumbersome when $K_n$ is large since it requires to solve  a system of
$2K_n+1$ equations for each value of the frequency $f$ where the function $L_n(f,\{\tilde{c}_k(f)\})$ should be evaluated.
In many cases (including the renewal case investigated below, see Lemmas~\ref{moments} and~\ref{lem:deviation}),
we can prove that, as $n$ goes to infinity, if the number of harmonics $K_n$ grows slowly enough (say, at a logarithmic rate),
the Gram matrix $G_n(f)$ is approximately  $G_n(f)  \approx n \mathrm{Id}_{2 K_n+1}$,
where $\mathrm{Id}_p$ denotes the $p \times p$ identity matrix;
this suggests to approximate $L_n(f,\{\tilde{c}_k(f)\})$ by
$f \mapsto \sum_{j=1}^n Y_j^2 -n\sum_{|k|\leq K_n}|\hat{c}_k(f)|^2$.
The minimization of this quantity is equivalent to maximizing the cumulated periodogram $\Lambda_n$
defined by~(\ref{eq:CumPEr}).
%Since it is much easier to compute in practice, we shall base our estimation criterion on $\Lambda_n$ instead of the profile likelihood $L_n(f,\tilde{\mathbf{c}}_n(f))$.

That is why we propose to estimate $f_0$ by $\hat{f}_n$ defined as follows,
\begin{equation}\label{eq:ftilden}
\Lambda_n(\hat{f}_n)=\sup_{f\in[\fmin,\fmax]}\Lambda_n(f)\; ,
\end{equation}
where $[\fmin,\fmax]$ is a given interval included in $(0,\infty)$.
Consider the following assumptions.
\begin{enumerate}[(H1)]
\item \label{assum:periodic-function} $s_{\star}$ is a real-valued periodic function defined
  on $\rset$ with finite fundamental frequency $f_{\star}$.
\item \label{assum:random-design} $\{X_j \}$ are the observation time instants, modeled as a renewal process, that is,
$X_j=\sum_{k=1}^j V_k $,
where $\{ V_k \}$ is a an i.i.d sequence of non-negative random variables with finite mean. In addition, for all
$\epsilon>0$, $\sup_{|t|\geq\epsilon} |\Phi(t)|<1$, where $\Phi$ denotes the
characteristic function of $V_1$,
\begin{equation}\label{eq:Phi}
\Phi(t)\eqdef \PE[\exp(\rmi t V_1)] \eqsp.
\end{equation}
\item \label{assum:distribution-noise} $\{ \varepsilon_j \}$ are i.i.d. zero mean Gaussian random variables with (unknown)
variance $\sigma_{\star}^2 > 0$ and are independent from the random variables $\{X_j\}$.
\item \label{assum:strong-spread-out} The distribution of $V_1$ has a non-zero absolutely continuous part with
respect to the Lebesgue measure.
\end{enumerate}
Recall that in (H\ref{assum:periodic-function}) the fundamental frequency is uniquely defined 
for non constant functions as follows: $T_{\star}=1/f_{\star}$ is the smallest $T>0$ such that
$s_{\star}(t+T)=s_{\star}(t)$ for all $t$.
All the possible frequencies of $s_{\star}$ are then $f_{\star}/l$,
where $l$ is a positive integer.
Note that the assumption made on the distribution of $V_1$ in~(H\ref{assum:random-design})
is a Cramer's type condition, which is weaker than~(H\ref{assum:strong-spread-out}).
%Assumption~(H\ref{assum:strong-spread-out}) ensures that, for any constant $c > 0$, the process
%$\{ \langle c^{-1} X_j \rangle,\,j\geq1\}$, where $\langle x \rangle$ denotes the fractional part of $x$, is ergodic and
%converges to the uniform distribution on the interval $[0,c]$.

The following result shows that $\hat{f}_n$ is a consistent estimator of the frequency contained by $[\fmin,\fmax]$
under very mild assumptions and give
some preliminary rates of convergence. These rates will be improved in Theorem~\ref{theo:mainCLT} under more
restrictive assumptions.

\begin{theo}\label{theo:mainCons}
Assume (H\ref{assum:periodic-function})--(H\ref{assum:distribution-noise}).
Let $\{K_n \}$ be a sequence of positive integers tending to infinity such that
\begin{equation}\label{eq:KnCondCons}
\lim_{n\to +\infty}K_n \left\{\rem(n^{\beta}) + n^{-1/2+\beta}\right\}=0\quad\text{for some $\beta>0$} \; ,
\end{equation}
% and
% \begin{equation}\label{eq:Knvarepsn}
% \lim_{n\to +\infty}K_n\rem(n^{\delta})= 0
% \quad\text{for some $\delta>0$}\; ,
% \end{equation}
where
\begin{equation}\label{eq:reste}
\rem(m)\eqdef \sum_{|k|>m} |c_k(s_{\star})| ,\quad m\geq0 \;.
\end{equation}
Let $\hat{f}_n$ be defined by~(\ref{eq:ftilden}) with $0<\fmin<\fmax$ such that $f_{0}$ is the unique
number $f\in[\fmin,\fmax]$ for which $s_{\star}$ is $1/f$--periodic. Then, for any $\alpha>0$,
\begin{equation}\label{eq:weakCons}
\hat{f}_n = f_{0} + o_{p}(n^{-1+\alpha})  \; .
\end{equation}
If we assume in addition that $\PE(V_1^2)$ is finite, then
\begin{equation}\label{eq:ConvPS}
n(\hat{f}_n - f_{0}) \to 0 \quad a.s.
\end{equation}
\end{theo}
% \begin{equation}\label{eq:KnCondCons}
% \lim_{n\to +\infty}K_n=+\inafty
% % \quad\text{and}\quad
% % \lim_{n\to +\infty} K_n\rem(n^{\delta})=0 \quad\text{for some}\quad \delta>0 \; .
% %K_n = O\left(n^{1/4-\epsilon}\right)
%\lim_{n\to +\infty}\kappa_n\frac{K_n}{\log n}=+\infty,\quad
%  \lim_{n\to +\infty}\kappa_n\frac{\sqrt{n}}{K_n \log n}=+\infty \;
% \end{equation}
% Let $(\kappa_n)$ be a sequence taking its values in $(0,1)$ and such that

\begin{proof}[Proof (sketch)]
Since $f_{\star}$ is the fundamental frequency of $s_\star$, the assumption on $[\fmin,\fmax]$ is equivalent to saying 
that there
exists a unique positive integer $\ell$ such that
\begin{equation}
  \label{eq:FminFmax}
\fmin \leq f_{\star}/\ell \leq \fmax \;,
\end{equation}
and that $f_{\star}/\ell=f_{0}$.
Using~(\ref{eq:Y}), we split $\Lambda_n$ defined in~(\ref{eq:CumPEr}) into three terms:
$\Lambda_n(f)=D_n(f)+\xi_n(f)+\zeta_n(f)$ where
\begin{align}
\label{eq:DnDef}
&D_n(f)=\frac{1}{n^2}\sum_{k=1}^{K_n}\left|\sum_{j=1}^n s_{\star}(X_j) \ \textrm{e}^{-2\textrm{i}k\pi f X_j}\right|^2 \; ,\\
\label{eq:XinDef}
&\xi_n(f)=\frac{2}{n^2}\sum_{k=1}^{K_n}\sum_{j,j'=1}^n\cos\{2\pi k f(X_j-X_{j'})\}\,s_{\star}(X_j) \varepsilon_{j'} \; ,\\
%\sum_{j=1}^n \left(\frac{1}{n^2}\sum_{k=1}^n \left\{\sum_{|l|\leq K_n , l\neq 0}
%\textrm{e}^{2\textrm{i}\pi l(X_j-X_k)f}\right\}s_{\star}(X_k)\right)\varepsilon_j
%\ \textrm{and}\
\label{eq:zetaDef}
&\zeta_n(f)=\frac{1}{n^2}\sum_{k=1}^{K_n}\left|\sum_{j=1}^n \varepsilon_j \ \textrm{e}^{-2\textrm{i}k\pi f X_j}\right|^2 \; .
\end{align}
We prove in Lemma~\ref{eta} of Section \ref{sec:proofs} that $\xi_n+\zeta_n$ tends uniformly to
zero in probability as $n$ tends to infinity. Then, by Lemmas~\ref{outballs} and~\ref{inballs}
proved in Section~\ref{sec:usef-interm-results}, for any $\alpha$, $D_n$ is maximal in balls centered at
sub-multiples of $f_{\star}$ with radii of order $n^{-1+\alpha}$ with probability tending to 1. Since,
by~(\ref{eq:FminFmax}), the interval $[\fmin,\fmax]$ contains but the sub-multiple $f_{\star}/\ell=f_{0}$, $\hat{f}_n$
satisfies~(\ref{eq:weakCons}). This line of
reasoning is detailed in Section~\ref{sec:proof-eqrefWeakCons}.
The obtained rate is then refined in \eqref{eq:ConvPS}
by adapting the consistency proof of~\citet*{quinn:thomson:1991} to our random design context (see Section~\ref{sec:proof-equat-eqrefConvPS}).
\end{proof}

\begin{remark}
One can construct a sequence $\{K_n \}$ satisfying Condition~(\ref{eq:KnCondCons}), as soon as
\begin{eqnarray}\label{eq:Wiener}
\sum_{k\in\zset} |c_k(s_{\star})|<+\infty \; ,
\end{eqnarray}
which is a very mild assumption.
%Under the additional assumption
%that there exists $\nu >0$ such that
%\begin{eqnarray}\label{eq:WienerPlus}
%\sup_{m\geq 0} (1+m)^\nu \sum_{|k|\geq m} |c_k(s_{\star})| <\infty \; ,
%\end{eqnarray}
%one may for instance choose any $\{K_n\}$ such that $(\log(n))^{1/(2\nu)}= o(K_n)$ and $K_n=o(n^{1/2-\beta})$ for some $\beta>0$
%and $\kappa_n\asymp1/\log(n)$.
\end{remark}
\begin{remark}\label{rem:consistencyfstar}
Observe that $f_0=f_{\star}$ in Theorem~\ref{theo:mainCons} if and only if $f_{\star}/2< \fmin\leq f_{\star} \leq \fmax$.
\end{remark}
\begin{remark}\label{rem:consistency}
If $\fmin$ and $\fmax$ are such that $s_\star$ has several frequencies in $[\fmin,\fmax]$, that is,~(\ref{eq:FminFmax}) has
multiple solutions $\ell=l_{\min},\dots,l_{\max}$, by partitioning
$[\fmin,\fmax]$ conveniently, we get instead of~(\ref{eq:weakCons}) (resp.~(\ref{eq:ConvPS})) that there exists a
random sequence $(\ell_n)$ with values in $\{l_{\min},\dots,l_{\max}\}$ such that $\hat{f}_n = f_{\star}/\ell_n + o_{p}(n^{-1+\alpha})$ (resp.
$n(\hat{f}_n - f_{\star}/\ell_n) \to 0$). This is not specific to our estimator. Unless an appropriate procedure is used to
select the largest
\emph{plausible} frequency, any  standard frequency estimators will in fact converge to a set of sub-multiples of $f_{\star}$.
% We refer to~\cite{gassiat:levyleduc:2006} or ~\cite{golubev:1988} REF OK????? where some procedures are proposed to estimate
% so that $f_\star$
% is consistently estimated.
\end{remark}

We now derive a Central Limit Theorem which holds for our estimator when
Condition~(\ref{eq:KnCondCons}) is strengthened into~(\ref{eq:sobolev})
and~(\ref{eq:KnCond}) and a finite fourth moment is assumed on $V_1$. 
%We allow $\hat{f}_n$ not to converge to $f_{\star}$ or
%to a sub-multiple of it, but to get close to a finite set of such sub-multiples; see  Remark~\ref{rem:consistency}.

\begin{theo}\label{theo:mainCLT}
Assume (H\ref{assum:periodic-function})--(H\ref{assum:strong-spread-out}).
Assume in addition that $\PE[V_1^4]$ is finite and that $s_{\star}$ satisfies
\begin{eqnarray}\label{eq:sobolev}
\sum_{k\in\zset} |k|^3 |c_k(s_{\star})|<+\infty \;.
\end{eqnarray}
Let $\{K_n\}$ be a sequence of positive integers tending to infinity such that
\begin{equation}\label{eq:KnCond}
\lim_{n\to +\infty} K_nn^{-\epsilon}=0\quad\text{for all}\quad \epsilon>0 \; .
\end{equation}
Let $\hat{f}_n$ be defined by~(\ref{eq:ftilden}) with $0<\fmin<\fmax$ such that $f_{0}$ is the unique
number $f\in[\fmin,\fmax]$ for which $s_{\star}$ is $1/f$--periodic.
%Let $(\hat{f}_n)$ be any random sequence such that $\dot{\Lambda}_n(\hat{f}_n)=o_p(\sqrt{n})$, where $\Lambda_n$ 
%is defined by~(\ref{eq:CumPEr}), and $\hat{f}_n-f_{n}=o_P(n^{-1})$, where $(f_n)$ is a random sequence taking 
%values in a finite set of frequencies of $s_\star$. 
Then we have the following asymptotic linearization:
%\begin{equation}\label{eq:LinearizationHatf}
%n^{3/2}(\hat{f}_n-f_{n})=\frac{\mu\;f_{n}}{n^{3/2}\sigma_{\star}^2 f_{\star}^2I_{\star}}\sum_{j=1}^n
%\left(j-\frac{n}{2}\right)\dot{s}_{\star}(X_j)\left(\varepsilon_j+s_{\star}(X_j)\right) + o_{p}(1)\;,
%\end{equation}
\begin{equation}\label{eq:LinearizationHatf}
n^{3/2}(\hat{f}_n-f_{0})=\frac{\mu\;f_{0}}{n^{3/2}\sigma_{\star}^2 f_{\star}^2I_{\star}}\sum_{j=1}^n
\left(j-\frac{n}{2}\right)\dot{s}_{\star}(X_j)\left(\varepsilon_j+s_{\star}(X_j)\right) + o_{p}(1)\;,
\end{equation}
where $\mu=\PE(V_1)$ and
\begin{equation}\label{eq:Istar}
I_{\star}=\frac{\mu^2}{12\sigma_{\star}^2
  f_{\star}}\int_0^{1/f_{\star}} \dot{s}_{\star}^2 (t) dt\; .
\end{equation}
%is the efficient Fisher information
Moreover $\hat{f}_n$ satisfies the following Central Limit Theorem
%\begin{eqnarray}\label{CLT}
%n^{3/2}(\hat{f}_n-f_{n})\frac{f_\star}{f_n}\stackrel{\mathcal{L}}{\longrightarrow}\mathcal{N}(0,\check{\sigma}^2),
%\end{eqnarray}
\begin{eqnarray}\label{CLT}
n^{3/2}(\hat{f}_n-f_{0})\frac{f_\star}{f_0}\stackrel{\mathcal{L}}{\longrightarrow}\mathcal{N}(0,\check{\sigma}^2),
\end{eqnarray}
where
\begin{equation}\label{eq:asymVar}
\check{\sigma}^2=I_{\star}^{-1}
\left\{1+\frac{\displaystyle\sum_{k\neq 0} |c_k(s_{\star}\dot{s}_{\star})|^2
\left(\frac{1-|\Phi(2k\pi f_{\star})|^2}{|1-\Phi(2k\pi f_{\star})|^2}\right)}
{\sigma_{\star}^2\sum_{k\in\zset}|c_k(\dot{s}_{\star})|^2}\right\},
\end{equation}
% with, for all $k$ in $\zset^{\star}$,
% $$|\tilde{c}_k(s_{\star}\dot{s}_{\star})|^2=|c_k(s_{\star}\dot{s}_{\star})|^2
% \left(\frac{1-|\Phi(2k\pi f_{\star})|^2}{|1-\Phi(2k\pi f_{\star})|^2}\right).
% $$
where $\Phi$, defined in~(\ref{eq:Phi}), denotes the characteristic function of $V_1$ and $\{c_k(s),k\in\zset\}$ denote the
Fourier coefficients of a $1/f_{\star}$-periodic function $s$ as defined in~(\ref{fourierCoeff}) with $T=1/f_{\star}$.
\end{theo}

\begin{proof}[Proof (sketch)]
%Since $(f_n)$ takes a finite number of values, one can define increasing sequences of integers whose union gives $\mathbb{N}$ and for
%which each corresponding subsequence of $(f_n)$ is constant.
%To prove~(\ref{eq:LinearizationHatf}) and~(\ref{CLT}), we may consider these subsequences separately so that we may assume
%that $(f_n)$ is constant without loss of generality. Hence we take $f_n=f_0$ for all $n$ in the following.

To derive (\ref{eq:LinearizationHatf}), we use a Taylor expansion of $\dot{\Lambda}_n(f)$, the first derivative
of $\Lambda_n(f)$ with respect to $f$, which provides
$$
\dot{\Lambda}_n(\hat{f}_n)=\dot{\Lambda}_n(f_{0})+(\hat{f}_n-f_{0})\ddot{\Lambda}_n(f'_n),
$$
where $f'_n$ is random and lies between $\hat{f}_n$ and $f_{0}$.
We prove in Sections~\ref{sec:proof-eq.-eqrefllambprim} and~\ref{sec:proof-eq.-eqrefllambsec} that
\begin{eqnarray}\label{lambprim}
\dot{\Lambda}_n(f_{0})=\sum_{j=1}^n \left(\frac{X_j}{n}-\frac{\mu}{2}\right)\frac{\dot{s}_{\star}(X_j)}{f_{0}}
\left(\varepsilon_j+s_{\star}(X_j)\right) +o_{p}(\sqrt{n}),\\
\label{lambsec}
\ddot{\Lambda}_n(f'_n)=-n^2 \frac{\mu^2}{12\;f_{0}}\int_0^{1/f_0} \dot{s}_{\star}^2(t)dt +o_{p}(n^2).
\end{eqnarray}
Since $\dot{s}_{\star}^2$ is $1/f_\star$-periodic and $f_\star/f_0$ is an integer, we have
$$
\frac{\mu^2}{12}\int_0^{1/f_0} \dot{s}_{\star}^2(t)dt=
\frac{\mu^2f_\star}{12\;f_{0}}\int_0^{1/f_{\star}} \dot{s}_{\star}^2 (t) dt=\frac{\sigma_\star^2f_\star^2I_\star}{f_0}
$$

The last three displayed equations and the assumption on $\dot{\Lambda}_n(\hat{f}_n)$ thus yield
$$
n^{3/2}(\hat{f}_n-f_{0})=\frac{f_0}{\sigma_\star^2f_\star^2I_\star\;n^{3/2}}
\sum_{j=1}^n
\left(X_j-\frac{n\mu}{2}\right)\dot{s}_{\star}(X_j)\left(\varepsilon_j+s_{\star}(X_j)\right) (1+ o_{p}(1)) + o_{p}(1)\;,
$$
and Relations~(\ref{eq:LinearizationHatf}) and \eqref{CLT} then follow from
\begin{align}\label{eq:Lincentrage}
&n^{-3/2}\sum_{j=1}^n
\left(X_j-j\mu\right)\dot{s}_{\star}(X_j)\varepsilon_j=o_p(1) \ \text{and}\
n^{-3/2}\sum_{j=1}^n\left(X_j-j\mu\right)\dot{s}_{\star}(X_j)s_{\star}(X_j)=o_p(1) \;,\\
  \label{eq:SnConv}
&S_n \eqdef \frac{\mu}{n^{3/2}\sigma_{\star}^2 f_{\star}I_{\star}}\sum_{j=1}^n
\left(j-n/2\right)\dot{s}_{\star}(X_j)\left(\varepsilon_j+s_{\star}(X_j)\right)
\stackrel{\mathcal{L}}{\longrightarrow}\mathcal{N}(0,\check{\sigma}^2)\;.
\end{align}
The proof of \eqref{eq:Lincentrage} follows from straightforward computations and is not detailed here.
We conclude with the proof of~(\ref{eq:SnConv}). By~(H\ref{assum:distribution-noise}), we have
$ S_n\stackrel{d}{=}  \frac{\mu}{\sigma_{\star}^2 f_{\star}I_{\star}} \left(A_n Z + U_n\right)$
where $A_n\eqdef n^{-3/2}\left(\sum_{j=1}^n \left(j-\frac{n}{2}\right)^2 {\dot{s}_{\star}}^2(X_j)\right)^{1/2}$,
$U_n={n^{-3/2}}\sum_{j=1}^n \left(j-n/2\right)(s_{\star}\dot{s}_{\star})(X_j)$
and $Z$ has distribution $\mathcal{N}(0,\sigma_{\star}^2)$
and is independent from the $X_j$'s. Therefore, since $Z$ and $U_n$ are independent,~(\ref{eq:SnConv}) follows from the two assertions
\begin{align}
\label{An}
& A_n=\left(\frac{1}{12}c_0({\dot{s}_{\star}}^2)\right)^{1/2} (1+o_p(1)) \; ;\\
\label{WnTCL}
&\frac{\mu}{\sigma_{\star}^2 f_{\star}I_{\star}} U_n
\stackrel{\mathcal{L}}{\longrightarrow} \mathcal{N}\left(0,\frac{\sum_{k\neq 0} |c_k(s_{\star}\dot{s}_{\star})|^2
\left(\frac{1-|\Phi(2k\pi f_{\star})|^2}{|1-\Phi(2k\pi f_{\star})|^2}\right)}{I_{\star}\,\sigma_{\star}^2\sum_{k\in\zset}
|c_k(\dot{s}_{\star})|^2}\right) \; .
\end{align}
Assertions \eqref{An} and \eqref{WnTCL} follow straightforwardly from (\ref{llnWeightMC}) and (\ref{eq:gammaG})
in Proposition \ref{Prop:CLT} respectively.
\end{proof}

\begin{remark}\label{RemExp}
If the $\{V_k\}$ are exponentially distributed (the sampling scheme is a Poisson process),
%then $\Phi(t)=\lambda/(\lambda-it)$ which
%yields the identity:
%$
%1-|\Phi(t)|^2= |1-\Phi(t)|^2 \; .
%$
then~(\ref{eq:asymVar}) yields
$
\check{\sigma}^2=
I_{\star}^{-1}
\left\{1+\|s_{\star}\dot{s}_{\star}\|_2^2/
(\sigma_{\star}^2\|\dot{s}_{\star}\|_2^2) \right\},
$
$\|\cdot\|_p$ denoting the usual $L^p$ norm on $[0,1/f_{\star}]$.
\end{remark}

%\section{Discussion}\label{sec:discussion}
In~\cite{gassiat:levyleduc:2006}, the local asymptotic normality (LAN) of the semiparametric model~(\ref{eq:Y})
is established for regular sampling with decreasing sampling instants. Their arguments can be extended to the
irregular sampling scheme. More precisely, any estimator satisfying the asymptotic linearization
\begin{equation}\label{eq:LinearizationOptimal}
n^{3/2}(\bar{f}_n-f_{\star})=\left(\frac{\mu}{n^{3/2}\sigma_{\star}^2 f_{\star}I_{\star}}\sum_{j=1}^n
\left(j-\frac{n}{2}\right)\dot{s}_{\star}(X_j)\varepsilon_j\right)(1+o_p(1))\;,
\end{equation}
where $I_\star$ is defined in~(\ref{eq:Istar}),  is an efficient semiparametric estimator of $f_{\star}$ in the sense of
\cite{macneney:wellner:2000}.
As a byproduct of the proof of Theorem~\ref{theo:mainCLT}, one has that the right hand-side of~(\ref{eq:LinearizationOptimal})
is asymptotically normal with mean zero and variance $I_{\star}^{-1}$. Hence $I_{\star}^{-1}$ is the optimal asymptotic
variance.
In view of~(\ref{eq:LinearizationHatf}) and~(\ref{CLT}), we see that the linearization of our estimator $\hat{f}_n$ contains
an extra term since $\varepsilon_j$ in~(\ref{eq:LinearizationOptimal}) is replaced by $(\varepsilon_j+s_{\star}(X_j))$
in~(\ref{eq:LinearizationHatf}). This extra term leads to an additional term in the asymptotic variance~(\ref{eq:asymVar}),
which highly depends on the distribution of the $V_k$'s.
Hence our estimator enjoys the optimal $n^{-3/2}$ rate but is not efficient.
The estimator proposed in~\cite{hall:reimann:rice:2000} is efficient and thus, in
theory, outperforms the CLSP estimator. On the other hand, our estimator is numerically more
tractable, and it does not require a preliminary consistent estimator. In contrast,
an interval containing the true frequency with size at
$o_p (n^{-(3/2-1/12)})$ is required in the assumptions of~\cite{hall:reimann:rice:2000} (see p.554 after conditions (a)--(e))
and, whether this assumption is necessary is an open question.
Nevertheless, since our estimator is rate optimal, it can be used as a preliminary estimator
to the one of~\citet*{hall:reimann:rice:2000}.

\section{Numerical experiments}\label{implementation}

Let us now  apply the proposed estimator to periodic variable stars
which are known to emit light whose intensity, or brightness, changes over
time in a smooth and periodic manner. The estimation of the period is of direct
scientific interest, for instance as an aid to classifying stars into different categories for
making inferences about stellar evolution. The irregularity in the observation
is often due to poor weather conditions and to instrumental constraints.

We benchmark the CLSP estimator with
the least-squares method (see \eqref{eq:ResidualSumofSquares}), which is reported as giving
the best empirical results in an extended simulation experiment which can be found in \cite{hall:reimann:rice:2000}.
In this Monte-Carlo experiment, we generate synthetic observations corresponding to model \eqref{eq:Y}
where the underlying deterministic function $s_{\star}$ is obtained by fitting a trigonometric polynomial
of degree 6 to the observations of a Cepheid variable star avalaible from the MACHO database
(\textsf{http://www.stat.berkeley.edu/users/rice/UBCWorkshop}). Figure \ref{fig1bis} displays
in its left part the observations of the Cepheid as points with coordinates
$(X_j~\textrm{modulo}~3.9861,Y_j)$, where $3.9861$ is the known period of the Cepheid 
and the $X_j$ are the observation times given by the MACHO
database. In the right part of Figure \ref{fig1bis}, the observations 
$(X_j,s_{\star}(X_j))$ are displayed as points with coordinates $(X_j~\textrm{modulo}~3.9861,s_{\star}(X_j))$ 
where $s_{\star}$ is a trigonometric polynomial of degree 6 fitted to the observations of the Cepheid variable
star that we shall use.
\begin{figure}[!ht]
\begin{center}
\includegraphics*[width=7cm]{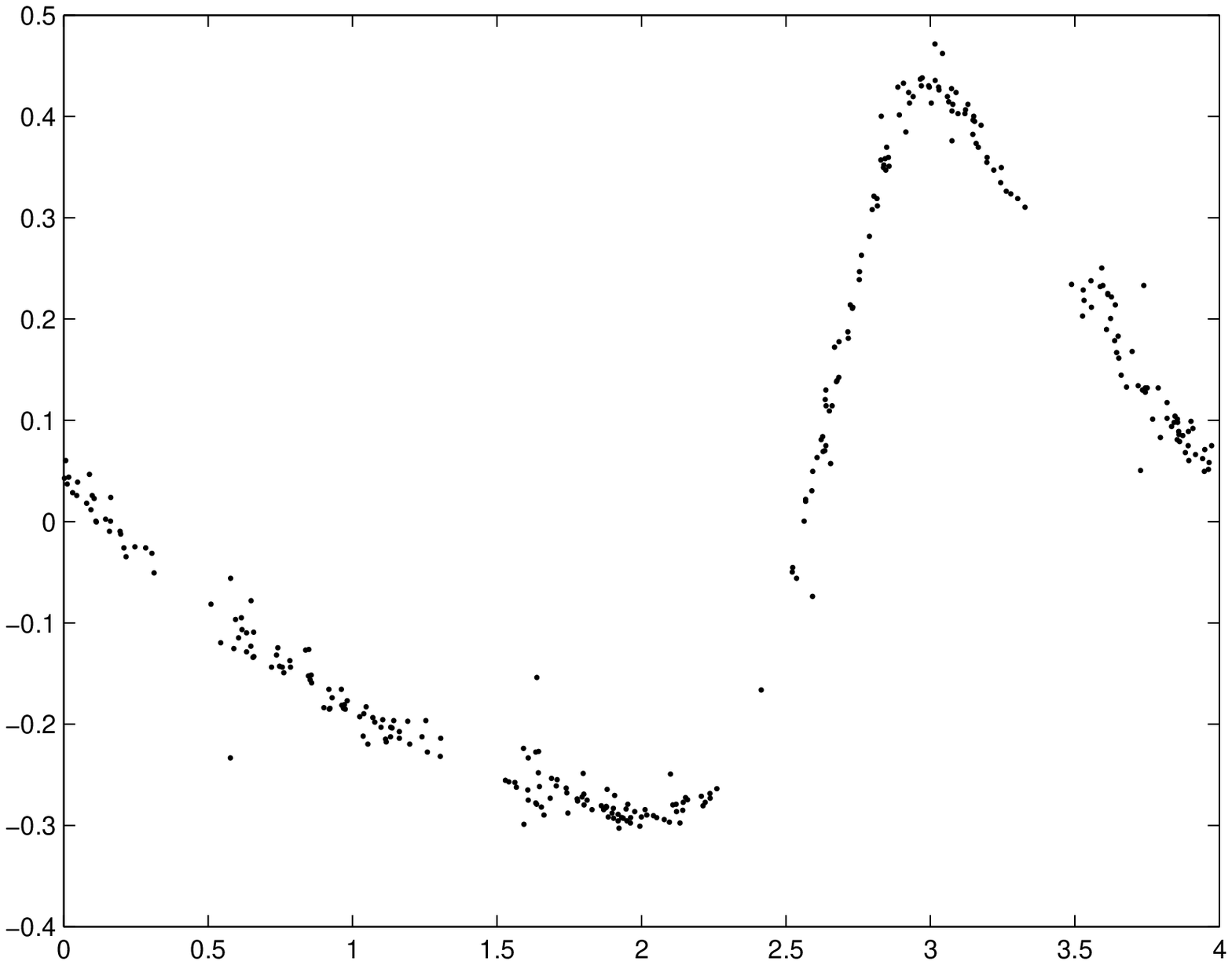}
\hspace{0.5cm}
\includegraphics*[width=7cm]{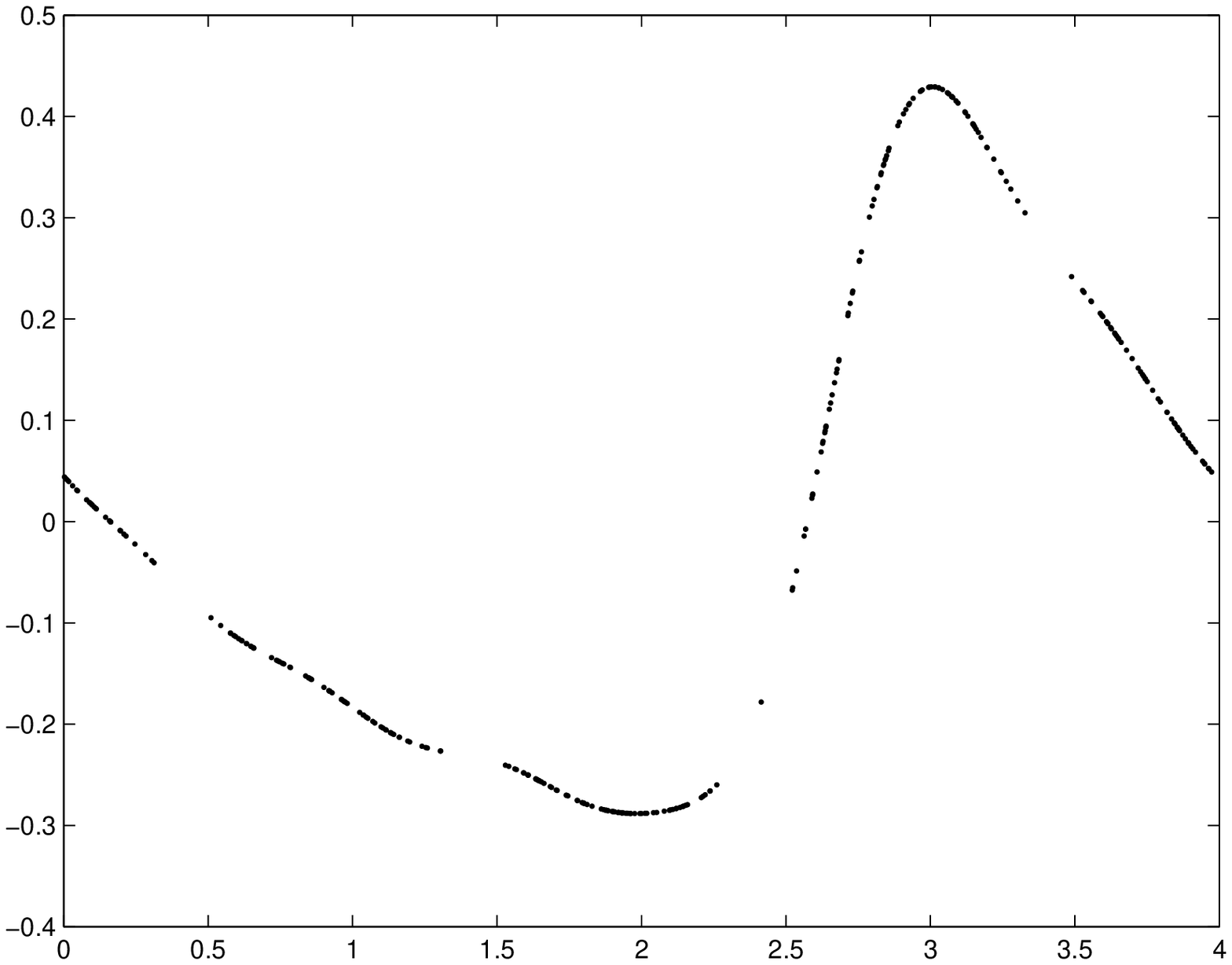}
\end{center}
\caption{Left: Cepheid observations, Right: Trigonometric polynomial $s_{\star}$ fitted to the Cepheid observations.}
\label{fig1bis}
\end{figure}

Let us now describe further the framework of our experiments.
The inter-arrivals $\{V_k\}$ have an exponential distribution 
with mean $1/5$.
The additive noise is i.i.d. Gaussian with standard deviations equal to 0.07 and 0.23 respectively
(the corresponding signal to noise ratios (SNR) are 10dB and 0dB). Typical
realizations of the observations that we process 
are shown in Figure \ref{fig2bis}  when $f_{\star}=0.25$ and $n=300$ in the two previous cases on the left and
right side respectively. More precisely, the observations $(X_j,Y_j)$ are displayed as stars
with coordinates $(X_j~\textrm{modulo}~1/f_{\star},Y_j)$ and the observations $(X_j,s_{\star}(X_j))$
of the underlying function $s_{\star}$ are displayed as points with coordinates:
$(X_j~\textrm{modulo}~1/f_{\star},s_{\star}(X_j))$. Since $E(V_1)=1/5$, $n=300$ and $1/f_{\star}=4$
approximately 15 periods are overlaid.
\begin{figure}[!ht]
\begin{center}
\includegraphics*[width=7cm]{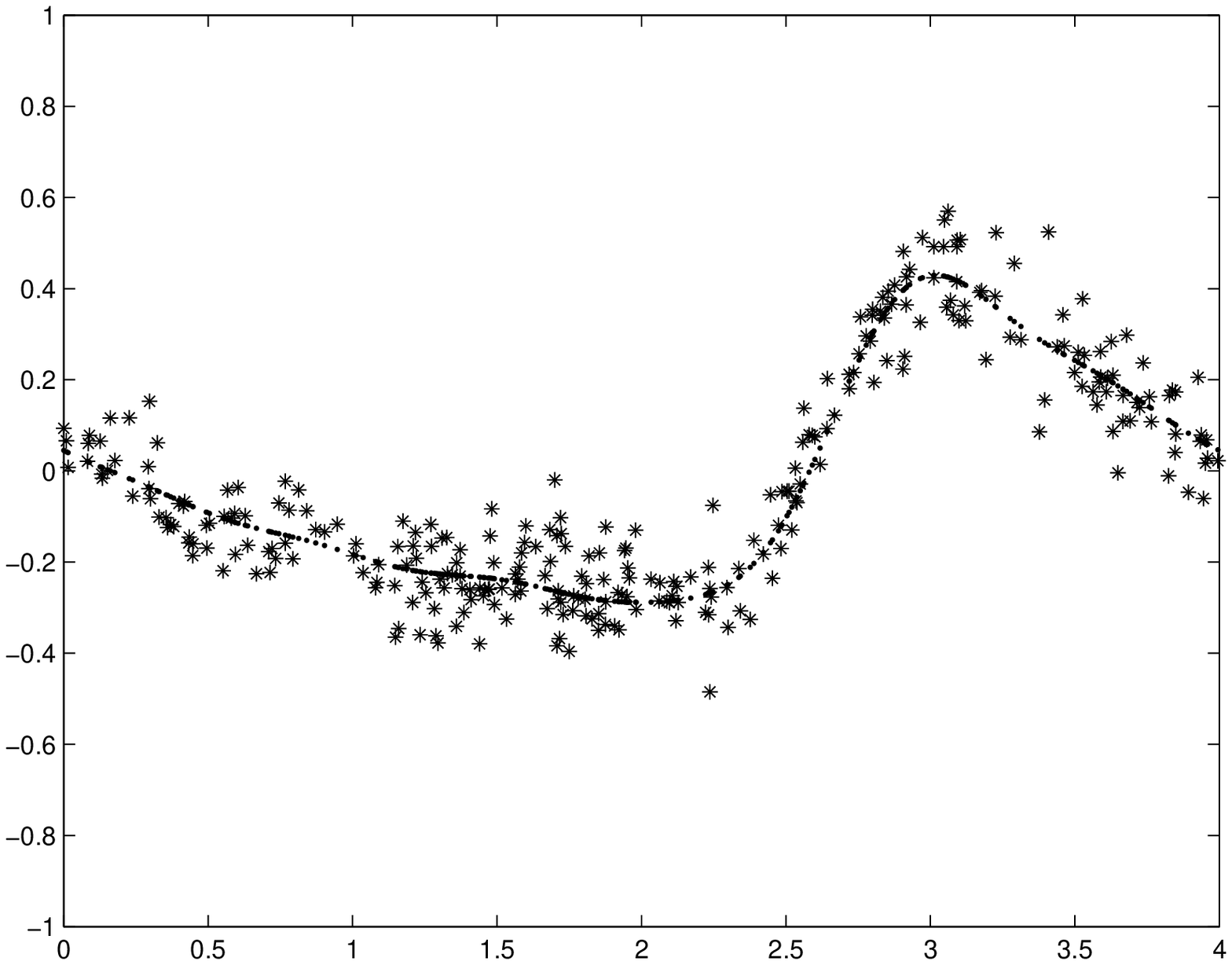}
\hspace{0.5cm}
\includegraphics*[width=7cm]{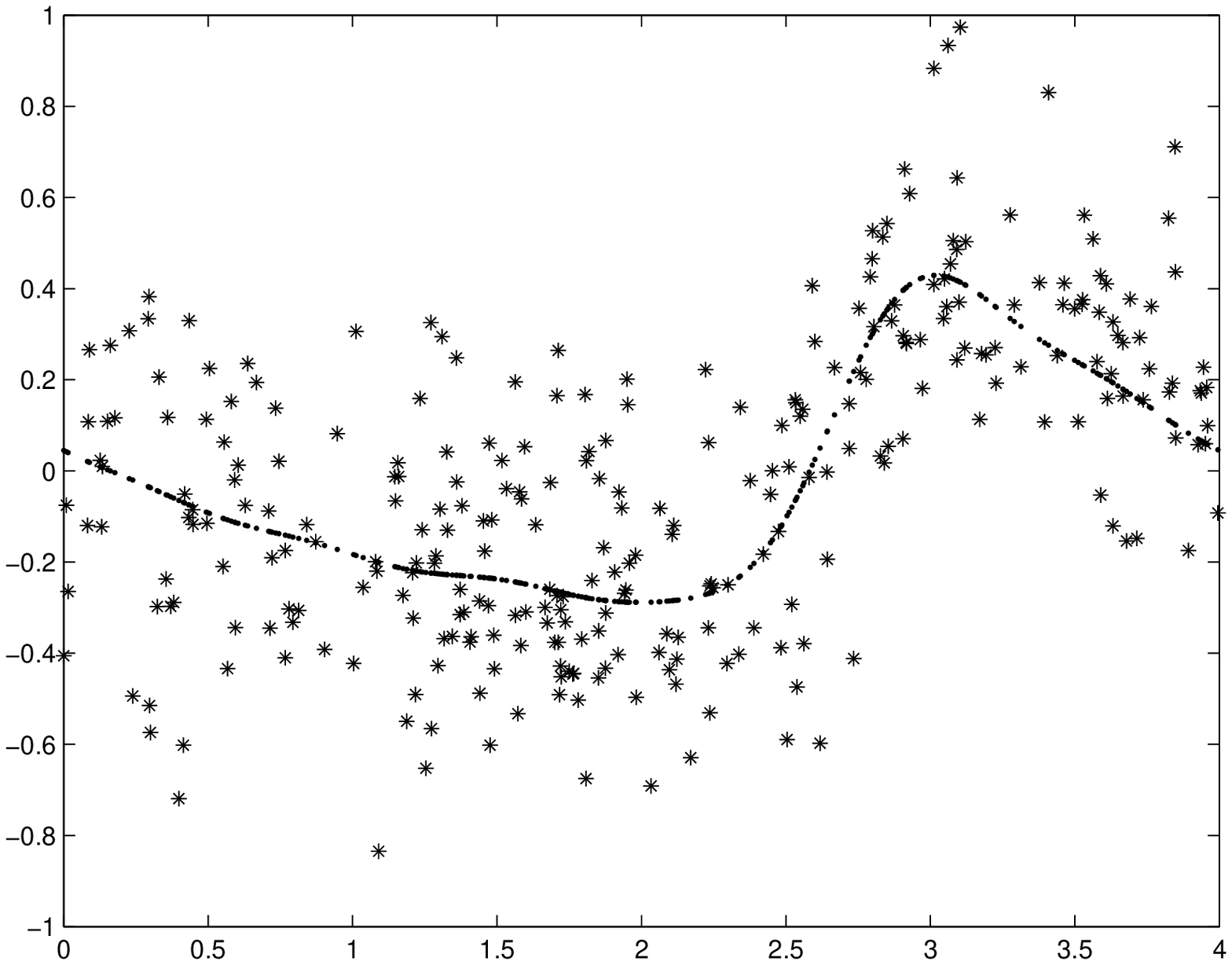}
\end{center}
\caption{Deterministic signal ('.') and noisy observations ('$\ast$') with SNR=10dB (left) and 0dB (right).}
\label{fig2bis}
\end{figure}

The least-squares and cumulated periodogram criteria ($L_n$ and $\Lambda_n$) are maximized on a grid
ranging from 0.2 to 0.52 with regular mesh $5\times 10^{-5}$.
Since the fundamental frequency is equal to 0.25,
the chosen range does not contain a sub-multiple of the
fudamental frequency but a multiple. Hence we are in the case $\ell=1$ in~(\ref{eq:FminFmax}), see
Remark~\ref{rem:consistency}. We used $K_n=1,2,4,6,8$.
The results of 100 Monte-Carlo experiments are summarized in Table \ref{tableFourier}.
We display the biases, the standard deviations (SD) and the optimal standard
deviations forecast by the theoretical study when $n=300, 600$ and SNR=10dB, 0dB.
\begin{table}[!ht]
\begin{footnotesize}
\hspace{-15mm}
\begin{tabular}{c | c c | c c | c c | c c}
&\multicolumn{4}{c}{SNR=10dB} & \multicolumn{4}{|c}{SNR=0dB} \\
\hline
& \multicolumn{2}{c|}{$n=300$}&\multicolumn{2}{c|}{$n=600$}&\multicolumn{2}{c|}{$n=300$} & \multicolumn{2}{c}{$n=600$}\\
\hline
Optimal SD &\multicolumn{2}{c|}{$1.21\times 10^{-4}$}&\multicolumn{2}{c|}{$4.28\times 10^{-5}$}
&\multicolumn{2}{c|}{$3.83\times 10^{-4}$} & \multicolumn{2}{c}{$1.35\times 10^{-4}$}\\
\hline
Method  & Bias$\times 10^{-5}$  & SD $\times 10^{-4}$ & Bias $\times 10^{-5}$ & SD$\times 10^{-4}$
& Bias$\times 10^{-4}$  & SD $\times 10^{-4}$ & Bias $\times 10^{-5}$ & SD$\times 10^{-4}$  \\
\hline
LS1    & 2.30 & 3.03    & 0.20   & 1.25               & 1.00 & 6.81  & 5.80   & 2.62    \\
CP1    & 7.70 & 5.13    & 3.28   & 2.11               & 0.00 & 8.20  & 7.20   & 2.91    \\
LS2    & 4.50 & 1.73    & 0.24   & 0.66               & 1.12 & 4.84  & 3.60   & 1.85   \\
CP2    & 5.90 & 4.09    &-0.16   & 1.76               & 0.49 & 6.88  & 2.50   & 2.12   \\
LS4    & 1.99 & 1.15    & 0.00   & 0.48               & 0.79 & 4.67  & 3.70   & 1.45   \\
CP4    & 1.89 & 3.83    &-3.08   & 1.58               & 0.31 & 6.07  & 1.99   & 2.13   \\
LS6    & 1.30 & 1.13    &-0.40   & 0.48               & 0.97 & 5.38  & 2.40   & 1.44   \\
CP6    &-2.20 & 4.23    &-2.96   & 1.75               & 0.37 & 7.58  & 2.70   & 2.22   \\
LS8    & 0.60 & 1.19    &-0.24   & 0.50               & 1.24 & 6.33  & 2.10   & 1.63   \\
CP8    &-1.30 & 5.28    &-2.32   & 1.75               & 0.79 & 9.02  & 4.40   & 2.72    \\
\hline
\end{tabular}
\caption{Biases, standard deviations of the frequency
  estimates for different methods: Least-squares (LS1, LS2, LS4, LS6, LS8), Cumulated
  Periodogram (CP1, CP2, CP4, CP6, CP8) with $K_n=1,2,4,6,8$.}
\label{tableFourier}
\end{footnotesize}
\end{table}

From the results gathered in Table \ref{tableFourier}, we get that the
least-squares estimator produces better results than the CLSP estimator.
Nevertheless, the CLSP estimator can be used as an accurate preliminary estimator 
of the frequency  since its computational cost is lower than the one of the 
least-squares estimator. For both estimators, the parameter $K_n$ has to be chosen 
carefully in order to achieve the best trade-off between bias and variance. 
Finding a way of choosing $K_n$ adaptively is left for future research.

\section{Detailed proofs}\label{sec:proofs}

In this section we provide some important intermediary results and we detail the arguments sketched in the proofs of Section~\ref{results}.

\subsection{Technical lemmas}\label{sec:techlem}

%\begin{lemma}\label{lem:weakerAsymp}
%(H\ref{assum:random-design}) is implied by (v).
%\end{lemma}
%\begin{proof}
%By the Jensen inequality, if $\Phi(t)=1$ for some $t\neq0$, then
%$\exp(\rmi t V)$ is constant a.s., in other words,
%$V$ takes its values in a set of the form $v_0+2\pi\zset/t$ a.s., which contradicts (v).
%Now, since the Fourier transform of any integrable function
%tends to zero at $\pm\infty$, (v) implies that
%$$
%\limsup_{t\to\infty}|\Phi(t)| <1 \; ,
%$$
%which concludes the proof.
%\end{proof}
The following Lemma provides upper bounds for the moments of  the empirical characteristic function of $X_1,\dots,X_n$,
\begin{equation}\label{eq:phinxDef}
\varphi_{n,X}(t)=\frac{1}{n}\sum_{j=1}^n \rme^{\textrm{i} t X_j} \eqsp.
\end{equation}
\begin{lemma}\label{moments}
Let~(H\ref{assum:random-design}) hold.
Then, for any non-negative integer $k$, there exists a positive constant $C$ such that for all $t\in\rset$,
%\note{EM}{il me semble que la borne $n^{k-1}(1+|t \vee n^{-1}|^{-1})$ est valide et dans ce cas là la
%preuve de la deuxième assertion est bcp plus rapide}
\begin{align}\label{eq:phinxDerkBounds1}
\left|\PE[\varphi_{n,X}^{(k)}(t)]\right|\leq C\,\max_{1 \leq l \leq k} \PE[V_1^l]\,
n^{k-1}(n\wedge |t|^{-1})\leq C\, (1+\PE(V_1^k))\, n^{k-1}(n\wedge |t|^{-1})\; , \\
\label{eq:phinxDerkBounds2}
\PE\left[\left|\varphi_{n,X}^{(k)}(t)\right|^2\right]\leq
\PE(V_1^{2k}) n^{2k-1}+
C (1+\PE(V_1^k))n^{2k-1}(n\wedge |t|^{-1})\;.
\end{align}
\end{lemma}

The proof of Lemma \ref{moments} is omitted since it comes from straightforward algebra.
The following Lemma provides an exponential deviation inequality for
$\varphi_{n,X}$ defined in~(\ref{eq:phinxDef}).

\begin{lemma}\label{lem:deviation}
Under Assumption (H\ref{assum:random-design}), we have, for all $x>0$ and $t\in\rset$,
\begin{eqnarray}\label{eq:ExpDev}
\PP\left(\left|\varphi_{n,X}(t)-\PE(\varphi_{n,X}(t))\right|\geq x\right)
\leq 4\exp\left(-\frac{n x^2\left|1-\Phi(t)\right|}{16(2+\sqrt{2})}\right),
\end{eqnarray}
where $\Phi$ is the characteristic function of $V$ defined in~(\ref{eq:Phi}).
\end{lemma}

\begin{proof}
Note that $\prod_{k=1}^j \textrm{e}^{\textrm{i} t V_k} -\Phi^j(t) =\sum_{q=1}^j \Phi^{j-q}(t) \Pi_q(t)$
where $\Pi_q(t) \eqdef \prod_{k=1}^{q} \textrm{e}^{\textrm{i} t V_k}-\Phi(t) \prod_{k=1}^{q-1} \textrm{e}^{\textrm{i} t V_k} \eqsp.$
Thus,
\begin{eqnarray*}
n\left(\varphi_{n,X}(t)-\PE\left[\varphi_{n,X}(t)\right]\right)=\sum_{j=1}^n
\left[\left(\prod_{k=1}^j \textrm{e}^{\textrm{i} t V_k}\right)-\Phi^j(t)\right]=
\sum_{j=1}^n\sum_{q=1}^j \Phi^{j-q}(t) \Pi_q(t)=\sum_{q=1}^n
\alpha_{n,q}(t) \Pi_q(t)
\end{eqnarray*}
where $\alpha_{n,q}(t)=\sum_{j=q}^n \Phi^{j-q}(t)=(1-\Phi(t))^{-1}(1-\Phi^{n-q+1}(t)),$
the last equality being valid as soon as $\Phi(t)\neq 1$. Let $\mathcal{F}_q$ denotes the $\sigma$-field generated by
$V_1,\dots,V_q$. Note that
$\left\{\alpha_{n,q}(t)\Pi_q(t),\ q\geq 1\right\}$ is a
martingale difference adapted to the filtration $(\mathcal{F}_q)_{q\geq1}$ and
$$
|\alpha_{n,q}(t)\Pi_q(t)|\leq 4|1-\Phi(t)|^{-1}\quad\text{and}\quad
\PE^{\mathcal{F}_{q-1}}\left[ |\alpha_{n,q}(t)\Pi_q(t)|^2 \right] \leq \frac{4(1-|\Phi(t)|^2)}{|1-\Phi(t)|^2} \leq  8|1-\Phi(t)|^{-1} \; ,
$$
where $\PE^{\mathcal{F}_q}$ denotes the conditional expectation  given $\mathcal{F}_q$. The proof then follows from
Bernstein inequality for martingales (see \cite{steiger:1969} or \cite{freedman:1975}).
%By adapting the Bernstein Inequality for real-valued martingale differences, see \citet*{bernstein} or \citet*{ottucsak}
%for a more recent reference, we obtain
%$$
%\PP\left(\left|\varphi_{n,X}(t)-\PE(\varphi_{n,X}(t))\right|\geq x\right)
%\leq 4\exp\left(-\frac{n x^2\left|1-\Phi(t)\right|}{32+8x\sqrt{2}/3}\right),
%$$
%This implies~(\ref{eq:ExpDev}) for $x\in(0,2]$. For $x>2$,~(\ref{eq:ExpDev}) holds trivially.
\end{proof}

For completeness, we state the following result, due to \citet*{golubev:1988}.
\begin{lemma}\label{golubev:1988}
Let $\mathcal{L}$ be a stochastic process defined on an interval  $I\subseteq\rset$.
Then, for all $\lambda, R >0$,
$$
\PP\left(\sup_{\tau\in I}\mathcal{L}(\tau)>R\right) \leq \rme^{-\lambda R}
\sup_{\tau\in  I}\left(\sqrt{\PE\left(\rme^{2\lambda\mathcal{L}(\tau)}\right)}\right)
\left(1+\lambda\int_{\tau\in I}\sqrt{\PE\left(\left|\dot{\mathcal{L}}(\tau)\right|^2\right)}d\tau\right).
$$
\end{lemma}

\subsection{Useful intermediary results}\label{sec:usef-interm-results}
We present here some intermediary results which may be of independent interest.

\begin{lemma}\label{eta}
Assume (H\ref{assum:random-design})-(H\ref{assum:distribution-noise}) and that $s_{\star}$ is bounded.
Define $\xi_n(f)$ and
$\zeta_n(f)$ by~(\ref{eq:XinDef}) and~(\ref{eq:zetaDef}) where $(K_n)$ is a sequence tending to infinity at most with a
polynomial rate. Then, for any  $0<\fmin<\fmax$, $\delta>0$ and $q=0,1,\dots$,
\begin{equation}\label{eq:UnifEta}
\sup_{f\in [\fmin , \fmax]} \left|\xi^{(q)}_n(f)+\zeta^{(q)}_n(f)\right|=o_p(K_n^{q+1} n^{q-1/2+\delta}) \;,
\end{equation}
where, for any function $h$, $h^{(q)}$ denotes the $q$-th derivative of $h$.
\end{lemma}
\begin{proof}
%We will use the following standard result on Laplace transform of Gaussian and quadratic forms of Gaussian random vectors,
%namely, if $Z$ is a centered Gaussian vector with identity covariance matrix, then for any real valued vector $x$,
%\begin{equation}\label{eq:LapGausVec}
%\PE\left[\exp(x^TZ)\right]=\exp(x^Tx/2)
%\end{equation}
%and, for any hermitian matrix $\Lambda$ having all its eigenvalues less than $1/4$,
%\begin{equation}\label{eq:LapGausQuad}
%\PE\left[\exp(Z^T\Lambda Z)\right]\leq\exp(\trace(\Lambda)+2\trace(\Lambda^2)) \; .
%\end{equation}
%
By~(\ref{eq:XinDef}) and~(\ref{eq:zetaDef}),
\begin{equation}
\xi^{(q)}_n(f)
%2\frac{(2\pi)^q}{n^2} \sum_{k=1}^{K_n}  \sum_{j,j'=1}^n k^q (X_j-X_{j'})^q \cos^{(q)}\left\{2k\pi f(X_j-X_{j'})\right\}s_{\star}(X_j)\varepsilon_{j'}
=2\frac{(2\pi)^q}{n^2} \sum_{k=1}^{K_n} k^q L_q(X,kf)^T\varepsilon \; , \quad
\label{eq:zetaQ}
\zeta^{(q)}_n(f)
%=\frac{(2i\pi)^q }{n^2}\sum_{k=1}^{K_n} \sum_{1\leq j,l\leq n} k^q (X_l-X_j)^q \ \rme^{2\rmi k\pi(X_l-X_j)f}\varepsilon_j \varepsilon_l
=\frac{(2\pi)^q}{n^2} \sum_{k=1}^{K_n} k^q \varepsilon^T \Gamma_q(X,kf)\varepsilon \; ,
\end{equation}
where $\varepsilon=[\varepsilon_1,\dots,\varepsilon_n]^T$, $ L_q(X,f) \eqdef \left[\sum_{j=1}^n
(X_j-X_{j'})^q \cos^{(q)}\left\{2\pi f(X_j-X_{j'})\right\}s_{\star}(X_j)\right]_{1\leq j'\leq n}
$ and $\Gamma_q(X,f) \eqdef \left[\rmi^q (X_l-X_j)^q\rme^{2\rmi\pi (X_l-X_j) f}\right]_{1\leq l,j\leq n}$.
Hence,
\begin{equation}\label{eq:ZetaNsansK}
\sup_{f\in[\fmin,\fmax]}\left|\xi^{(q)}_n(f)+\zeta^{(q)}_n(f)\right|\leq
C\,n^{-2}\,K_n^{q+1}\sup_{0<f\leq K_n\fmax}
\left|2 L^T_q(X,f)\varepsilon + \varepsilon^T \Gamma_q(X,f)\varepsilon\right| \; .
\end{equation}
Note that $\trace[\Gamma_q(X,f)]=0$ for $q\geq 1$, $\trace[\Gamma_q(X,f)]=n$ for $q=0$
and that the spectral radius of the matrix $\Gamma_q(X,f)$ is at most
$\sup_{j=1,\dots,n}\sum_{l=1}^n |X_l-X_j|^q\leq n X_n^q$. For any hermitian matrix $\Lambda$ having all its eigenvalues less than $1/4$,
$\PE\left[\exp(Z^T\Lambda Z)\right]\leq\exp(\trace(\Lambda)+2\trace(\Lambda^2))$.
Therefore, for any $\lambda>0$, on the event $ \left\{ \lambda\sigma^{\star\,2}  n X_n^q\leq1/8 \right\}$,
\begin{multline*}%\label{eq:LapetaSecond}
\CPE{\rme^{\lambda (2 L_q(X,f)^T\varepsilon + \varepsilon^T \Gamma_q(X,f)\varepsilon)}}{X}
\leq C'\exp\left\{C \lambda^2 \left(L_q(X,f)^TL_q(X,f)+ \trace(\Gamma_q^2(X,f))\right)\right\}\nonumber\\
\leq C'\exp\left\{C \lambda^2 n^3 X_n^{2q}\right\} \; ,
\end{multline*}
where we have used $L^T_q(X,f) L_q(X,f) \leq C n^3 X_n^{2q}$ and $\trace[ \Gamma_q^2(X,f) ] \leq C n^2 X_n^{2q}$.
Using~(H\ref{assum:distribution-noise}), we similarly get that
$\CPE{|L_{q}(X,f)^T\varepsilon|^2}{X} \leq C L_{q}^T(X,f) L_{q}(X,f) \leq Cn^3 X_n^{2q} $
and
\begin{equation}
  \label{eq:quadBoundcondX}
 \CPE{|\varepsilon^T \Gamma_{q}(X,f)\varepsilon|^2}{X} \leq C \trace[ \Gamma^2_{q}(X,f) ] \leq C\,n^2X_n^{2q} \; .
\end{equation}
Applying Lemma~\ref{golubev:1988}, we get that, for all positive
numbers $\lambda$ and $R$, on the event $\{ \lambda\sigma^{\star\,2}  n X_n^q\leq1/8 \}$,
\begin{multline*}
\CPP{\sup_{0<f\leq K_n\fmax} \left|2 L_q(X,f)^T\varepsilon + \varepsilon^T \Gamma_q(X,f) \varepsilon\right|\geq R}{X}
\leq \\ C'\, \rme^{-\lambda R +C\lambda^2n^3 X_n^{2q}}\,\left(1+C\,K_n\,\lambda\,n^{3/2}\,X_n^{q+1}\right) \; .
\end{multline*}
Let $\delta>0$. Applying this inequality with $\lambda=n^{-3/2}X_n^{-q}$ and $R=n^{\delta+3/2}X_n^q$, we get
$$
\CPP{\sup_{0<f\leq K_n\fmax} \left|2 L_q(X,f)^T\varepsilon + \varepsilon^T \Gamma_q(X,f) \varepsilon\right| \geq n^{\delta+3/2}X_n^q}{X}\leq C\, \exp(-n^\delta)\,\left(1+K_n\,X_n\right) \; .
$$
Now, using~(\ref{eq:ZetaNsansK}),
$\PP\left(\sup_{f\in [\fmin , \fmax]} \left|\xi_n^{(q)}(f)+\zeta_n^{(q)}(f)\right|\geq n^{q-1/2+2\delta}K_n^{q+1}\right)$
%and conditioning on the event  $\{X_n^q \geq n^{q+\delta}\}$,
\begin{eqnarray*}
%\leqP\left(\sup_{0<f\leq K_n\fmax}\left|2 L_X(f)^T\varepsilon + \varepsilon^T \Gamma_X(f)\varepsilon\right|\geq  n^{q+3/2+2\delta}\right)\\
&\leq& \PP\left(\sup_{0<f\leq K_n\fmax}\left|2 L_X(f)^T\varepsilon + \varepsilon^T \Gamma_X(f)\varepsilon\right|\geq
n^{\delta+3/2}X_n^q\right) + \PP\left(X_n^q \geq n^{q+\delta}\right)  \\
&\leq&  C\, \exp(-n^\delta)\,\left(1+K_n\,n\right)  + n^{-\delta/q} \; ,
\end{eqnarray*}
which concludes the proof.
\end{proof}

Let us introduce the following notation. For some sequence $(\gamma_n)$, and $f_{\star}>0$, we define, for any positive integers
$n,j$ and $l$,
\begin{multline}\label{eq:bnj_lDef}
B_n(j,l) \eqdef [\fmin,\fmax]\cap\{f~:~ |f-jf_{\star}/l|\leq \gamma_n\}
\textrm{ and }  B_n^c(j,l) \eqdef [\fmin,\fmax]\setminus B_n(j,l)\; .
\end{multline}

\begin{lemma}\label{outballs}
Assume (H\ref{assum:periodic-function})--(H\ref{assum:random-design})  and that $s_{\star}$ satisfies~(\ref{eq:Wiener}).
Define $D_n(f)$ by~(\ref{eq:DnDef}), with a sequence $(K_n)$ tending to infinity.
Then, for any $\epsilon>0$, as $n$ tends to infinity,
$$
\sup_{f\in\bigcap_{j,l}B_n^c(j,l)}
D_n(f)=O_p\left(K_n \rem(m_n)^2+ \frac{\{K_n\,n\,m_n\}^\epsilon}{n\gamma_n}\right) \; ,
$$
where $(m_n)$ is a sequence of positive integers, $\rem$ is defined by~(\ref{eq:reste}),
$\bigcap_{j,l}$ is the intersection over integers $j\geq1$, $l=1,\dots,K_n$ and $B_n^c(j,l)$ is defined by~(\ref{eq:bnj_lDef})
with $0<\fmin<\fmax$ and $(\gamma_n)$ satisfying
\begin{equation}\label{eq:outballsGamman}
K_n\gamma_n\to 0\quad\text{and}\quad n\gamma_n \to\infty \; .
\end{equation}
\end{lemma}
\begin{proof}
We use the Fourier expansion~(\ref{fourierCoeff}) of $s_{\star}$ defined with the minimal period $T=1/f_\star$.
Expanding $s_\star$ in~(\ref{eq:DnDef}) and using the definition of $\varphi_{n,X}$ in~(\ref{eq:phinxDef}) and of $\rem(m)$
in~(\ref{eq:reste}), we get
\begin{eqnarray}
\nonumber
D_n(f)
%&=&\frac{1}{n^2}\sum_{k=1}^{K_n}\left|\sum_{p\in\zset} c_p(s_{\star})
%\left(\frac{1}{n}\sum_{j=1}^n \textrm{e}^{2\textrm{i}\pi X_j(pf_{\star}-kf)}\right)\right|^2
%=\sum_{k=1}^{K_n}\left|\sum_{p\in\zset} c_p(s_{\star})
%\varphi_{n,X}\left\{2\pi(pf_{\star}-kf)\right\}\right|^2\\
%&=&\sum_{k=1}^{K_n}\left|\sum_{|p|\leq K_n} c_p(s_{\star})
%\varphi_{n,X}\left\{2\pi(pf_{\star}-kf)\right\}+\sum_{|p|> K_n} c_p(s_{\star})
%\varphi_{n,X}\left\{2\pi(pf_{\star}-kf)\right\}\right|^2\\
&\leq&  2 \sum_{k=1}^{K_n}\left|\sum_{|p|\leq m}
    c_p(s_{\star})\varphi_{n,X}\left\{2\pi(pf_{\star}-kf)\right\}\right|^2
+  2  K_n\rem(m)^2  \\
\label{eq:DnDecomp}
& \leq& 4  \sum_{k=1}^{K_n}\left|\sum_{|p|\leq m} c_p(s_{\star})
  \PE\left[\varphi_{n,X}\left\{2\pi(pf_{\star}-kf)\right\}\right]\right|^2 + 4\tilde{D}_{n,m}(f) +  2  K_n\rem(m)^2 \; ,
\end{eqnarray}
where we defined
\begin{equation}\label{eq:tildeD}
\tilde{D}_{n,m}(f) \eqdef \sum_{k=1}^{K_n}\left|\sum_{|p|\leq m} c_p(s_{\star})
  \left(\varphi_{n,X}\left\{2\pi(pf_{\star}-kf)\right\}-\PE\left[\varphi_{n,X}\left\{2\pi(pf_{\star}-kf)\right\}\right]\right)\right|^2 \; .
\end{equation}
%In the following, $C$ denotes a positive constant not depending on $f$.
For all positive integers $j$ and $l\leq
K_n$, and $f\in \mathsf{A}_n \eqdef \bigcap_{j',l'}B_n^c(j',l')$, $2\pi|jf_{\star}-lf|\geq2\pi l \gamma_n$.
Thus, using~(\ref{eq:phinxDerkBounds1}) with $k=0$ in Lemma~\ref{moments},~(\ref{eq:Wiener}), and $\lim_{n \to \infty} n \gamma_n = \infty$,
we get, for all $f\in \mathsf{A}_n$, and $n$ large enough,
\begin{equation}\label{eq:tildeDCentrage}
\sum_{k=1}^{K_n}\left|\sum_{|p|\leq m} c_p(s_{\star})
  \PE\left[\varphi_{n,X}\left\{2\pi(pf_{\star}-kf)\right\}\right]\right|^2
\leq \frac C{n^2}\sum_{k=1}^{K_n} \frac1{k^2\gamma^2_n} = O\left(n^{-2}\gamma_n^{-2}\right) \; .
\end{equation}
Consider now $\tilde{D}_{n,m}$. For $\rho > 0$ and $q=1, \dots, Q(\rho) \eqdef [\rho^{-1}(\fmax-\fmin)]$, define
$I_q \eqdef [\fmin +  (q-1) \rho, \fmin+ q \rho] \cap \mathsf{A}_n$.
Observe that $u \mapsto \varphi_{n,X}(u)$ is a Lipschitz function  with Lipschitz norm less than
$n^{-1}\sum_{j=1}^n X_j$ and bounded by 1. It follows that $f \mapsto \tilde{D}_{n,m}(f)$ is a Lipschitz function
with Lipschitz norm less than
$$
8\pi\left(\sum_{p}|c_p(s_{\star})|\right)^2 \sum_{k=1}^{K_n} k n^{-1} \sum_{j=1}^n (X_j+\PE[X_j]) \leq
C\,\left\{\frac{K_n^2}{n}\sum_{j=1}^n X_j + n K_n^2 \right\} \;.
$$
Thus, for any $q=1,\dots,Q(\rho)$ such that $I_q$ is non-empty, and any $f_q \in I_q$,
$
\sup_{f\in I_q} \tilde{D}_{n,m}(f) \leq \tilde{D}_{n,m}(f_q) + C\,\rho\,\left\{K_n^2 n^{-1}\sum_{j=1}^n X_j + n K_n^2
\right\}  \; ,
$
which implies
\begin{equation}\label{eq:splitdtildeNQ}
\sup_{f\in \mathsf{A}_n}\tilde{D}_{n,m}(f) \leq \sup_{q=1,\dots,Q(\rho)}\tilde{D}_{n,m}(f_q) + O_p\left(\rho nK_n^2\right) \; ,
\end{equation}
where, by convention, $\tilde{D}_{n,m}(f_q)=0$ if $I_q$ is empty.
%Now, let $f\in\mathsf{A}_n$, $1\leq k\leq K_n$ and $p\in\zset$ and $C$ be a positive constant independent of these quantities.
Since by~(H\ref{assum:random-design}), $\inf_{t\in\rset}\left|1-\Phi(t)\right|/(1\wedge |t|) >0,$ and for $n$ large enough, $K_n\gamma_n\leq1$,
Lemma~\ref{lem:deviation} shows that,  for any $f \in \mathsf{A}_n$,  $2\pi|pf_{\star}-kf|\geq2\pi k\gamma_n$, and $y>0$,
$$
\PP\left(\left|\varphi_{n,X}\{2\pi(pf_{\star}-kf)\}-\PE[\varphi_{n,X}\{2\pi(pf_{\star}-kf)\}]\right|\geq y\right)
\leq 4\rme^{-C n y^2 (1\wedge k\gamma_n)}\leq4\rme^{-C n y^2 k\gamma_n}\;.
$$
Using this bound with the definition of $\tilde{D}_{n,m}(f)$ in~(\ref{eq:tildeD}), we get, for all $x>0$ and
$f\in\mathsf{A}_n$,
\begin{align*}
\PP\left(\sup_{q=1,\dots,Q(\rho)}\tilde{D}_{n,m}(f_q) \geq x\right) \leq Q(\rho)\sup_{f\in\mathsf{A}_n} \PP\left(\tilde{D}_{n,m}(f)\geq x\right)
%P\left(\tilde{D}_{n,m}(f)\geq x\right)
\leq 4Q(\rho) \sum_{k=1}^{K_n}\sum_{|p|\leq m} \rme^{-C n x\beta_k\alpha_p^2 k \gamma_n} \;,
\end{align*}
where $\beta_k,\,\alpha_p,\,k=1,\dots,K_n,\,|p|\leq m$ are positive weights such that $\sum_k\beta_k=1$ and
$\sum_p\alpha_p|c_p|=1$.
With $\beta_k=k^{-1}/(\sum_{k\leq K_n} k^{-1})\geq 2k^{-1}/\log(K_n)$ for $n$ large enough, and $\alpha_p=\left(\sum_{p}
  |c_p|\right)^{-1}\geq C$, we get
\begin{equation}\label{eq:DevtildeDnm}
\PP\left(\sup_{q=1,\dots,Q(\rho)}\tilde{D}_{n,m}(f_q) \geq x\right) \leq 4Q(\rho)
K_n(2m+1)\rme^{-C n x\gamma_n/\log(K_n)} \; .
\end{equation}
Let $\delta>0$. Defining $\rho_n\eqdef (n^{-2}\gamma^{-1}_nK_n^{-2}\log(K_n))$ and $x= (Q(\rho_n)K_nm_n)^\delta \log(K_n) /(n\gamma_n)$, implying $Q(\rho_n)\to\infty$ and
$Q(\rho_n) K_nm_n\to\infty$, we obtain
$$
\sup_{f\in\mathsf{A}_n}\tilde{D}_{n,m_n}(f) = o_p\left( \{n^2\gamma_nK_n^3m_n\}^\delta \log(K_n) \{n\gamma_n\}^{-1} \right) \; .
$$
For any $\epsilon>0$, we set $\delta>0$ small enough such that $\{n^2K_n^2m_n\}^\delta\log(K_n)=O(\{K_n\,n\,m_n\}^\epsilon)$.
The previous bound, with~(\ref{eq:outballsGamman}),~(\ref{eq:DnDecomp}),~(\ref{eq:tildeDCentrage})
and~(\ref{eq:splitdtildeNQ}) yields the result.
\end{proof}
\begin{lemma}\label{inballs}
Assume (H\ref{assum:periodic-function})--(H\ref{assum:random-design}) and that $s_{\star}$ satisfies~(\ref{eq:Wiener}). Define
$D_n(f)$ by~(\ref{eq:DnDef}), with a sequence $(K_n)$ tending to infinity.
Then, as $n$ tends to infinity, for all relatively prime integers $j$ and $l$,
\begin{equation}\label{eq:upperboundUprim}
D_n(jf_{\star}/l)=\sum_{k=1}^{[K_n/l]} |c_{kj}(s_{\star})|^2+O_p\left(K_n \, l^{1/2}\, n^{-1/2}\right) \;.
\end{equation}
Moreover, for any $\epsilon>0$,
\begin{multline}\label{eq:UnifupperboundUprim}
\sup_{(j,l)\in\mathcal{P}_n}\sup_{f\in B_n(j,l)}
\left|D_n(f) - \sum_{k=1}^{[K_n/l]} |c_{kj}(s_{\star})\varphi_{n,X}\left\{2\pi(kjf_{\star}-klf)\right\}|^2 \right|\\
= O_p\left(K_n\rem(m_n)+ K_n^2n^{-1}+ (K_n\,n\,m_n)^\epsilon K_n^{1/2} n^{-1/2} \right) \; ,
\end{multline}
where $(m_n)$ is a sequence of positive integers, $\rem$ is defined by~(\ref{eq:reste}),
$\mathcal{P}_n$ is the set of indices $(j,l)$ such that $j\geq1$ and $1\leq l\leq K_n$ are relatively prime integers and
$B_n(j,l)$ is defined by~(\ref{eq:bnj_lDef}) with $0<\fmin<\fmax$
$(\gamma_n)$ satisfying
\begin{equation}\label{eq:inballsGamman}
 \gamma_nK_n^2\to0 \; .
\end{equation}
\end{lemma}
\begin{proof}
Let  $j$ and $l\leq K_n$ be two relatively prime integers.  In the following, $C$ denotes a positive constant independent of
$j$, $l$ and $f$ that may change upon each appearance.
As in the proof of  Lemma~\ref{outballs}, we use the Fourier expansion~(\ref{fourierCoeff}) of $s_{\star}$ defined with
$T=1/f_{\star}$. Expanding $s_\star$ in~(\ref{eq:DnDef}),
the leading term in $D_{n}(f)$ for $f$ close to $jf_{\star}/l$ will be given by the indices $k$
and $p$ such that $k/l$ and $p/j$ are equal to the same integer, say $q$. Thus we split $D_{n}(f)$ into
\begin{equation}\label{eq:SplitAn}
D_{n}(f)=\sum_{q=1}^{[K_n/l]}\left|c_{qj}(s_{\star})
  \varphi_{n,X}\left\{2\pi(qjf_{\star}-qlf)\right\}\right|^2 + A_n(f) \;,
\end{equation}
where
\begin{equation*}%\label{eq:An}
A_n(f) \eqdef {\sum_{k,p,p'}}' c_p(s_{\star})\overline{c_{p'}(s_{\star})} \varphi_{n,X}\{2\pi(pf_{\star}-kf)\}
\overline{\varphi_{n,X}\{2\pi(p'f_{\star}-kf)\}}
\end{equation*}
with $\sum_{k,p,p'}'$ denoting the sum over indices $k=1,\dots,K_n$ and $p,p'\in\zset$ such that, for
any integer $q$, we have $k\neq ql$, $p\neq jq$ or $p'\neq jq$. It follows from this definition and from~(\ref{eq:Wiener}),
since $|\varphi_{n,X}|\leq1$, that
\begin{equation}\label{eq:AnBound}
|A_n(f)|
%\leq {\sum_{k,p,p'}}' |c_p(s_{\star})\overline{c_{p'}(s_{\star})} \varphi_{n,X}\{2\pi(pf_{\star}-kf)\}
%\overline{\varphi_{n,X}\{2\pi(p'f_{\star}-kf)\}} |
\leq  C\,{\sum_{k,p}}'|c_p(s_{\star})\varphi_{n,X}\{2\pi(pf_{\star}-kf)\}| \; ,
\end{equation}
where $\sum_{k,p}'$ denotes the sum over indices $k=1,\dots,K_n$ and $p\in\zset$ such that, for
any integer $q$, we have $k\neq ql$ or $p\neq jq$.
Using that $j$ and $l$ are relatively prime, if for any integer $q$, $k\neq ql$ or $p\neq jq$, then $|pl-kj|\geq1$, which
implies, by~\eqref{eq:phinxDerkBounds2}  with $k=0$ in Lemma \ref{moments},
$$
\PE\left[\left|\varphi_{n,X}\left\{2\pi(pf_{\star}-kf)\right\} \right|\right] \leq
\PE\left[\left|\varphi_{n,X}\left\{2\pi(pf_{\star}-kf)\right\} \right|^2\right]^{1/2}\leq C\, (n^{-1} l)^{1/2}  \; .
$$
Hence, using~(\ref{eq:Wiener}) and this bound in~(\ref{eq:AnBound}), Relation~(\ref{eq:SplitAn}) yields~(\ref{eq:upperboundUprim}).
We now proceed in bounding $A_n(f)$ uniformly for $f\in \cup_{(j,l)\in\mathcal{P}_n}B_n(j,l)$. We use the same line of reasoning as for
bounding $D_n(f)$ in Lemma~\ref{outballs}. First we split the sum in $p$ appearing in~(\ref{eq:AnBound}) and introduce
the centering term $\PE[\varphi_{n,X}\{2\pi(pf_{\star}-kf)\}]$ so that
\begin{equation}\label{eq:AnBoundUnif}
|A_n(f)| \leq C\,\left( A_{n,m}(f) + {\sum_{k,p}}' \left|c_p(s_{\star})\,\PE[\varphi_{n,X}\{2\pi(pf_{\star}-kf)\}]\right|
+ K_n \rem(m) \right) \; ,
\end{equation}
where
$$
A_{n,m}(f) \eqdef  {\sum_{k,p}}'' |c_p(s_{\star})(\varphi_{n,X}\{2\pi(pf_{\star}-kf)\}-\PE[\varphi_{n,X}\{2\pi(pf_{\star}-kf)\}])|  \; ,
$$
with $\sum_{k,p}''$ denoting the sum over indices $k=1,\dots,K_n$ and $p=0,\pm1,\dots,\pm m$ such that $|pl-kj|\geq1$.
Using~(\ref{eq:phinxDerkBounds1}) with $k=0$ in Lemma~\ref{moments} and~(\ref{eq:Wiener}), we have
\begin{equation}\label{eq:AnBoundUnifCentrage}
{\sum_{k,p}}' \left|c_p(s_{\star})\,\PE[\varphi_{n,X}\{2\pi(pf_{\star}-kf)\}]\right| \leq C\, l \, K_n n^{-1} \; .
\end{equation}
As for obtaining~(\ref{eq:splitdtildeNQ}), we cover $[\fmin,\fmax]$ with $Q$ intervals of size $\rho=(\fmax-\fmin)/Q$, and
obtain
$$
\sup_{f\in \cup_{(j,l)\in\mathcal{P}_n}B_n(j,l)}A_{n,m}(f) \leq \sup_{q=1,\dots,Q} A_{n,m}(f_q) + O_p\left(\rho n K_n^2\right) \; ,
$$
where either $A_{n,m}(f_q)=0$, or $f_q\in \bigcup_{j,l}B_n(j,l)$, in which case, for all indices $k$ and $p$ in the
summation term $\sum_{k,p}''$, there exist integers $j$ and $l\leq K_n$ such that
\begin{equation*}%\label{eq:minoration}
|pf_{\star}-kf_q|\geq |pf_{\star}-kjf_{\star}/l|- \gamma_n k \geq f_{\star}/l - \gamma_nK_n
\geq f_{\star}/K_n - \gamma_nK_n\geq C \, K_n^{-1} \; ,
\end{equation*}
for $n$ large enough, by~(\ref{eq:inballsGamman}).
Now, we apply the deviation estimate in Lemma~\ref{lem:deviation}, so that, as in~(\ref{eq:DevtildeDnm}), we have
$$
\PP\left(\sup_{q=1,\dots,Q} A_{n,m}(f_q) >  x \right) \leq 4 Q K_n (2m+1) \,\rme^{-C n x^2K_n^{-1}} \; .
$$
Let $\delta>0$. Setting $Q=[K_n^{3/2} n^{3/2}]$ and $x=(Q K_n m_n)^\delta K_n^{1/2}\,n^{-1/2}$  so that $Q\to\infty$ and
$Q K_n m_n\to\infty$ as $n\to\infty$,
we finally obtain
$$
\sup_{f\in \bigcup_{j,l}B_n(j,l)}A_{n,m}(f)=O_p\left((Q K_n m_n)^\delta K_n^{1/2}n^{-1/2}\right) \; .
$$
For any $\epsilon>0$, we set $\delta>0$ such that $(Q K_n m_n)^\delta=O((K_n\,n\,m_n)^\epsilon)$.
Applying this bound in~(\ref{eq:AnBoundUnif}) and using~(\ref{eq:AnBoundUnifCentrage}), Relation~(\ref{eq:SplitAn})
yields~(\ref{eq:UnifupperboundUprim}).
\end{proof}
 The following Proposition gives some limit results for additive functionals of a renewal
process.

\begin{prop}\label{Prop:CLT}
Assume (H\ref{assum:random-design}) and (H\ref{assum:strong-spread-out}). Let $g$ be a non-constant locally integrable
$T$-periodic real-valued function defined on $\rset$.
Assume that the Fourier coefficients of $g$ defined by~(\ref{fourierCoeff}) satisfy
$c_0(g)=0$ and $\sum_{k\in\zset} |c_k(g)| < \infty$ then for any non-negative
integer $k$
\begin{equation}\label{llnWeightMC}
\frac{1}{n^{k+1}}\sum_{j=1}^n j^k g(X_j)=O_p(n^{-1/2})\;.
\end{equation}
Denote by $s_n(t)$ the piecewise linear interpolation
$$
s_n(t) =  \sum_{k=1}^{[nt]} g(X_k) + (nt-[nt]) g(X_{[nt]+1}) ,\, t\geq0  \eqsp,
$$
where $[x]$ denotes the integer part of $x$. Then, as $n\to \infty$,
\begin{equation}\label{eq:gammaG}
(n\gamma_g^2)^{-1/2} s_n(t) \Rightarrow B(t) \;,\quad\text{where}\quad
\gamma_g^2 \eqdef \sum_{k\in\zset\backslash\{0\}}  \left|c_k(g)\right|^2 \,
\frac{1- \left| \Phi(2\pi k/T)\right|^2}{\left|1-\Phi(2\pi k/T)\right|^{2}}
\end{equation}
is positive and finite,
$\Rightarrow$ denotes the weak convergence in the space of continuous $[0,1]\to\rset$ functions endowed with the
uniform norm and $B(t)$ is the standard Brownian motion on $t\in[0,1]$.
\end{prop}

\begin{proof}
Without loss of generality we set $T=1$ in this proof section.
Define the Markov chain $\{Y_k\}_{k\geq0}$, valued in $[0,1)$ and started at $x\in[0,1]$ by $Y_0=x$
and $Y_{k+1}=Y_k+V_{k+1}-[Y_k+V_{k+1}]$, $k\geq0$.
Observe that, with the initial value $x=0$, we have $g(Y_k)=g(X_k)$ for all $k\geq1$.
Let us show that this Markov Chain is positive Harris and that its invariant probability  is the uniform distribution on
$[0,1]$. We first prove that this chain is uniformly Doeblin, for a definition see
\cite{cappe:moulines:ryden:2005}.
By (H\ref{assum:strong-spread-out}), there exists a non-negative and bounded function $h$ such that
$0<\int_0^\infty h(t)dt<\infty$  and for all Borel set $A$, $\PP(V\in A)\geq \int_A h(t) \, dt$.
It follows that, for any $k\geq1$, $\PP(X_k\in A)\geq \int_A h^{\ast k}(t) \, dt$,
where $h^{\ast k}=h\ast \dots \ast h$ ($k$ times) with $\ast$ denoting the convolution. Observe that the properties of $h$
imply that $h^{\ast 2}$ is non-negative, continuous and non--identically zero.
It follows that there exists $0\leq a<b$ and $\delta>0$ such that $\int_{t \in [a,b]} h^{\ast 2}(t)\geq\delta$. Hence, for
$k$ large enough, there exists a non-negative integer $l$ and $\epsilon>0$ such that
$h^{\ast (2k)}(t)=(h^{\ast 2})^{\ast k}(t)\geq\epsilon$ for
all $t\in[l,l+1]$. Hence, for all $x\in[0,1)$ and all Borel set $A\subset[0,1]$,
\begin{equation*}
\PP_x(Y_{2k}\in A) \geq  \PP_x(Y_{2k}\in A, X_{2k}\in[l,l+1)) \
%& = \PP_x(x+X_{2k}-l\in A, X_{2k}\in[l,l+1-x)) + \PP_x(x+X_{2k}-l-1\in A, X_{2k}\in[l+1-x,l+1])) \\
%& \geq  \epsilon \mathrm{Leb}(A+l-x\cap [l,l+1-x))  + \epsilon \mathrm{Leb}(A+l+1-x\cap [l+1-x,l+1))  \geq   \epsilon \mathrm{Leb}(A) \;,
\end{equation*}
which is the uniform Doeblin condition. This implies that $Y$ is a uniformly geometrically
ergodic Markov chain; let us
compute its invariant probability distribution, denoted by $\pi$. For all $x\in[0,1]$ and $l\in\zset$, $l\neq0$, we have
$
\PE_x[\exp(2\rmi\pi lY_n)]=\exp(2\rmi\pi x) \left(\Phi(2\pi l)\right)^n \to 0 \;,
$
where we used (H\ref{assum:random-design}) which is implied by (H\ref{assum:strong-spread-out}).
Hence, for all $l\in\zset$, $l\neq0$,
$ \int_{t=0}^1 \exp(2\rmi\pi l t) \pi(dt) = 0$, which implies that $\pi$ is the uniform distribution on $[0,1]$.
Define
$$
\tilde{g}(x)=\sum_{k\in\zset\backslash\{0\}} c_k(g) (1-\Phi(2\pi k))^{-1} \rme^{2\rmi\pi kx} \; .
$$
By (H\ref{assum:random-design}), $(1-\Phi(2\pi k))^{-1}$ is bounded uniformly on $k\in\zset\backslash\{0\}$.
Hence $\gamma_g$ is positive and finite.
Moreover,
$\sum_{k\in\zset\backslash\{0\}} |c_k(g) (1-\Phi(2\pi k))^{-1}| <\infty$ and we compute
$$
\PE_x[\tilde{g}(Y_1)]= \sum_{k\in\zset\backslash\{0\}} c_k(g)\frac{\Phi(2\pi k)}{(1-\Phi(2\pi k))}\exp(2\rmi\pi k x).
$$
This yields that $\tilde{g}$ is the solution of the Poisson equation
$\tilde{g}(x)-E_x[\tilde{g}(Y_1)]=g(x) - \int_0^1 g(t) dt$.
We now prove \eqref{llnWeightMC}. Note that, since $\pi(g)=0$,
\begin{multline*}
n^{-(k+1)}\sum_{j=1}^n j^k g(X_j)
=n^{-(k+1)}\sum_{j=1}^n j^k \left(\tilde{g}(X_j)-P\tilde{g}(X_j)\right)\\
=n^{-(k+1)}\sum_{j=1}^n j^k \left(\tilde{g}(X_j)-P\tilde{g}(X_{j-1})\right)
+n^{-(k+1)}\sum_{j=1}^n j^k
\left(P\tilde{g}(X_{j-1})-P\tilde{g}(X_{j})\right)\;.
\end{multline*}
Since $\tilde{g}$ is bounded, the variance of the first term is $O(n^{-1})$
as $n\to\infty$. Integrating by parts yields, using that $\tilde{g}$ is bounded,
$
n^{-(k+1)}\sum_{j=1}^n j^k \left(P\tilde{g}(X_{j-1})-P\tilde{g}(X_{j})\right)
=n^{-(k+1)}P\tilde{g}(X_{0})-n^{-1}P\tilde{g}(X_{n})+n^{-(k+1)}\sum_{j=1}^n \left[(j+1)^k-j^k\right]P\tilde{g}(X_{j})
=O_p(n^{-1}).
$
To prove \eqref{eq:gammaG} we compute, by the Parseval Theorem,
\begin{multline*}
\int_{0}^1 \left\{\tilde{g}^2(x)- (\PE_x[\tilde{g}(Y_1)])^2\right\} \, dx  \\
 = \sum_{k\in\zset\backslash\{0\}}
\left\{ \left|c_k(g) (1-\Phi(2\pi k))^{-1}\right|^2 -  \left| c_k(g) \Phi(2\pi k) (1-\Phi(2\pi k))^{-1}\right|^2
\right\}
%= \sum_{k\in\zset\backslash\{0\}}  \left|c_k(g)\right|^2 \,\frac{1- \left| \Phi(2\pi k)\right|^2}{\left|1-\Phi(2\pi k)\right|^{2}}
=\gamma_g^2 \; ,
\end{multline*}
The end of the proof follows from the functional central limit theorem \cite[Theorem 17.4.4]{meyn:tweedie:1993}.
\end{proof}

\subsection{Proof of  \eqref{eq:weakCons}}\label{sec:proof-eqrefWeakCons}

Let $\alpha >0$ arbitrary small and denote by $\{\gamma_n\}$ the sequence
\begin{equation}\label{eq:gamma_nTheo}
\gamma_n=n^{-1+\alpha} \; .
\end{equation}
Since $\ell$ is the unique integer satisfying~(\ref{eq:FminFmax}), for $n$ large enough, we have
$B_n(1,l)\cap[\fmin,\fmax]=\emptyset$ for all $l\neq\ell$, where $B_n(1,l)$ is defined by \eqref{eq:bnj_lDef}.
Hence, for $n$ large enough, $\PP(\hat{f}_n\notin B_n(1,\ell))\leq P_1+P_2$, where
$$
P_1=\PP\left(\sup\limits_{f\in\;\bigcap_{j,l} B_n^c(j,l)}\Lambda_n(f)\geq {\Lambda}_n(f_{0}/\ell)\right)
\text{ and }
P_2 = \PP\left(\sup\limits_{f\in\;\bigcup'_{j,l} B_n(j,l)}\Lambda_n(f)
\geq \Lambda_n(f_{0}/\ell)\right),
$$
where $\bigcap_{j,l}$ is the same as in Lemma~\ref{outballs} and
$\bigcup'_{j,l}$ the union over all $j\geq2$
and $l=1,\dots,K_n$ such that $j$ and $l$ are relatively prime. To show \eqref{eq:weakCons}, we thus need to show that
$P_1,P_2\to0$ as $n\to\infty$. Note that
\begin{equation}\label{eq:P1}
P_1\leq\PP\left(\sup_{f \in\;\bigcap_{j,l} B_n^c(j,l)} D_n(f)+2\sup_{f\in[\fmin,\fmax]} |\xi_n(f)+\zeta_n(f)|
\geq  D_n(f_{0}/\ell)\right).
\end{equation}
By~(\ref{eq:KnCondCons}), applying Lemma~\ref{eta}, we get
\begin{equation}\label{eq:SupEta}
\sup_{f\in[\fmin,\fmax]} |\xi_n(f)+\zeta_n(f)|=o_p(n^{-\beta/2}) \; .
\end{equation}
We now apply Lemma~\ref{outballs}.
Using~(\ref{eq:KnCondCons}) again and choosing $\alpha$ small enough in~(\ref{eq:gamma_nTheo}), we have $K_n\gamma_n\to0$
and, since $n\gamma_n\to\infty$ and $K_n\to\infty$, Condition~(\ref{eq:outballsGamman}) holds.
By~(\ref{eq:KnCondCons}) we have $K_n(n^{-1/2+\beta}+\rem(n^\beta)^2)\to0$ and, by~(\ref{eq:gamma_nTheo}), taking
$m_n=n^\beta$ and $\epsilon$ small enough in Lemma~\ref{outballs}, we obtain
$
\sup_{f \in\;\bigcap_{j,l} B_n^c(j,l)} D_n(f)=o_p(1).
$
The last two displays show that the left-hand side of the inequality in~(\ref{eq:P1}) converges to zero
in probability. Concerning its right-hand side $D_n(f_{0}/\ell)$,
Relation~(\ref{eq:upperboundUprim}) with $j=1$ and $l=\ell$ in Lemma~\ref{inballs} shows that, as $n\to\infty$,
 \begin{equation}\label{eq:RHSP1}
D_n(f_{0}/\ell) \stackrel{p}{\longrightarrow} \sum_{k\geq1}|c_k(s_{\star})|^2>0 \; .
\end{equation}
Hence $P_1\to0$. As in~(\ref{eq:P1}), we have
$$
P_2\leq \PP\left(\sup_{f\in \bigcup'_{j,l}B_n(j,l)}D_n(f)+ 2 \sup_{f\in[\fmin,\fmax]} |\xi_n(f)+\zeta_n(f)|
\geq D_n(f_{0}/\ell) \right) \; .
$$
To prove that $P_2\to0$, we use the following classical inequality, see~\cite{golubev:1988}
or \cite{gassiat:levyleduc:2006},
\begin{equation}\label{eq:S1}
\sup_{j\geq2} \sum_{k=1}^{\infty} |c_{kj}(s_{\star})|^2 < \sum_{k=1}^{\infty} |c_{k}(s_{\star})|^2 \; ,
\end{equation}
which directly follows from the fact that $f_{0}$ is the maximal
frequency of $s_{\star}$.
Now, we apply Lemma~\ref{inballs}. Using~(\ref{eq:KnCondCons}) , Condition \eqref{eq:inballsGamman} holds by choosing
$\alpha$ small enough in~(\ref{eq:gamma_nTheo}).
By~(\ref{eq:KnCondCons}), we have $K_n(n^{-1/2+\beta}+\rem(n^\beta))\to0$ and,
by~(\ref{eq:gamma_nTheo}), taking
$m_n=n^\delta$ and $\epsilon$ small enough in~(\ref{eq:UnifupperboundUprim}), we obtain, using~(\ref{eq:S1})
%$
%\PP\left(\sup_{f\in \bigcup'_{j,l}B_n(j,l)}D_n(f) \geq \sum_{k=1}^{\infty} |c_{k}(s_{\star})|^2 \right)\to 0 \; .
%$
and~(\ref{eq:RHSP1}), that $P_2\to0$, which concludes the proof.

\subsection{Proof of Eq.~\eqref{eq:ConvPS}}\label{sec:proof-equat-eqrefConvPS}

Let us first prove that, for any $\epsilon>0$,
\begin{equation}\label{eq:Fejer}
\sup_{|t|\leq n^{-1/2-\epsilon}} \left||\varphi_{n,X}(t)|^2-\frac1n F_n(\mu t)\right| = o_p(1) \;,
\end{equation}
where $\mu=\PE[V_1]$, $F_n(x) \eqdef \frac1n\left|\sum_{k=1}^n\rme^{\rmi kt}\right|^2$ is the Fejer kernel and
$\varphi_{n,X}$ is defined in~(\ref{eq:phinxDef}). Indeed, using a standard Lipschitz argument and
(H\ref{assum:random-design}) with the assumption $\PE[V_1^2]<\infty$,
$$
\PE\left[\sup_{|t|\leq n^{-1/2-\epsilon}} \left||\varphi_{n,X}(t)|^2-\frac1n F_n(\mu t)\right|\right]
\leq 2 n^{-1/2-\epsilon} \frac1n\sum_{k=1}^n\PE[|X_k- k \mu|] \leq 2 \sqrt{\Var(V_1)} n^{-\epsilon} \; ,
$$
which gives~(\ref{eq:Fejer}).
Now, by definition of $\hat{f}_n$, we have
$
0\leq \Lambda_n(\hat{f}_n)-\Lambda_n(f_{0}/\ell).
$
Beside, we have, using~(\ref{eq:SupEta}),
$$
\Lambda_n(\hat{f}_n)-\Lambda_n(f_{0}/\ell) \leq D_n(\hat{f}_n)-D_n(f_{0}/\ell)+2\sup_{f\in[\fmin,\fmax]}|\xi_n(f)+\zeta_n(f)|
=D_n(\hat{f}_n)-D_n(f_{0}/\ell)+o_p(1)
$$
and, since the event $\{\hat{f}_n\in B_n(1,\ell)\}$ has probability tending to one,
Lemma~\ref{inballs} yields, for $\alpha$ small enough
in~(\ref{eq:gamma_nTheo}),
$
D_n(\hat{f}_n)-D_n(f_{0}/\ell)\leq
\sum_{k=1}^{K_n}|c_k(s_{\star})|^2 [|\varphi_{n,X}\{2\pi k(f_{0}-\ell\hat{f}_n)\}|^2-1] +
o_p(1).
$
Hence, since for $\alpha$ small enough $K_n\gamma_n\leq n^{-1/2-\alpha/2}$, the last three displayed equations
and~(\ref{eq:Fejer}) finally yield that,
$
0\leq \sum_{k=1}^{K_n}|c_k(s_{\star})|^2 \left[\frac1n F_n\{2\pi \mu k(f_{0}-\ell\hat{f}_n)\}-1\right] +o_p(1).
$
We conclude the proof like in~\cite[Theorem 1, P. 68]{quinn:thomson:1991} by observing that, for any $c>0$,
$
\limsup_{n\to\infty} \sup_{|t|>c/n} \frac1nF_n(t) < 1\; .
$

\subsection{Proof of Eq.~\eqref{lambprim}}\label{sec:proof-eq.-eqrefllambprim}

We use that $\dot{\Lambda}_n(f_{0})=\dot{\xi}_n(f_{0})+\dot{\zeta}_n(f_{0})+\dot{D}_n(f_{0})$ so that~(\ref{lambprim}) follows
from
\begin{align}
\label{eq:XiPrime}
\dot{\xi}_n(f_{0}) & = \frac1{f_{0}}\sum_{j=1}^n \dot{s}_{\star}(X_j)
\left(\frac{X_j}{n}-\frac{\mu}{2}\right)\varepsilon_j  + o_p(\sqrt{n}) \; ,\\
\label{ZetaPrime}
\dot{\zeta}_n(f_{0}) & =o_p(\sqrt{n}) \; ,\\\label{DePrime}
\dot{D}_n(f_{0})& =\frac{1}{f_{0}}\sum_{j=1}^n\left(\frac{X_j}{n}-\frac{\mu}{2}\right)\dot{s}_{\star}(X_j)
s_{\star}(X_j)+o_{p}(\sqrt{n}) \; ,
\end{align}
which we now prove successively. Differentiating~\eqref{eq:XinDef}, we obtain
% Note that $\xi_n$ defined in \eqref{eq:XinDef} can be rewritten
% as follows
% $$
% \xi_n(f)=\frac{1}{n^2}\sum_{j,j'=1}^{n}\left(\sum_{|k|\leq K_n} \rme^{2\rmi\pi k(X_{j}-X_{j'})f}\right)
% s_{\star}(X_{j'})\varepsilon_{j}
% $$
% and thus
$\dot{\xi}_n(f_{0})=n^{-1}\sum_{j=1}^n A_n(j)\varepsilon_j$
where
$$
A_n(j)\eqdef n^{-1}\sum_{j'=1}^n\sum_{|k|\leq K_n} 2\rmi\pi k(X_{j}-X_{j'})\rme^{2\rmi\pi k(X_{j}-X_{j'})f_{0}}
s_{\star}(X_{j'}).
$$
%$$
%\dot{\xi}_n(f_{\star})=\frac{1}{n}\sum_{j=1}^n
%\left\{\frac{1}{n}\sum_{j'=1}^n\sum_{|k|\leq K_n} 2\rmi\pi k(X_{j}-X_{j'})\rme^{2\rmi\pi k(X_{j}-X_{j'})f_{\star}}
%s_{\star}(X_{j'})\right\}\varepsilon_{j}\;.
%$$
In this proof section, we use the Fourier expansion~(\ref{fourierCoeff}) defined with $T=1/f_{0}$.
Expanding $s_{\star}(X_{j'})$ and using the definition of $\varphi_{n,X}$ in \eqref{eq:phinxDef}, we
obtain for any $j=1,\dots,n$,
$
A_n(j)=\sum_{|k|\leq K_n}\rme^{2\rmi\pi k X_{j} f_{0}} \sum_{p\in\zset}  c_p(s_{\star})
(2\rmi\pi k)\left\{ X_{j}\ \varphi_{n,X}[2\pi(p-k)f_{0}]
+\rmi \dot{\varphi}_{n,X}[2\pi(p-k)f_{0}]\right\}.
$
In the sequel, we denote  $\overline{X_n}\eqdef n^{-1}\sum_{j=1}^n X_j$ and $\|Y\|_2=\PE(|Y|^2)^{1/2}$. By Minkowski's
inequality,
$
\dot{\xi}_n(f_{0})-(n f_{0})^{-1}\sum_{j=1}^n \dot{s}_{\star}(X_j)
\left(X_j-n\mu/2\right)\varepsilon_j=O_p\left(n^{-1}
\sum_{k=1}^3\left\{\sum_{j=1}^n
\|A_{n,k}(j)\|_{2}^2\right\}^{1/2}\right),
$
where
\begin{align*}
&A_{n,1}(j)=-f_{0}^{-1} \dot{s}_{\star}(X_j)(\overline{X_n}-n\mu/2) \eqsp, \\
&A_{n,2}(j)=-\sum_{|k|>K_n} (2\rmi\pi k) c_k(s_{\star})\rme^{2\rmi\pi k f_{0} X_j} (X_j-\overline{X_n}) \eqsp, \\
&A_{n,3}(j)=\sum_{|k|\leq K_n} \sum_{p\neq k}(2\rmi\pi k)  c_p(s_{\star}) \rme^{2\rmi\pi k X_j f_{0}}\left(X_j\ \varphi_{n,X}[2\pi(p-k)f_{0}] +\rmi\dot{\varphi}_{n,X}[2\pi(p-k)f_{0}]\right) \eqsp.
\end{align*}
Note that for all $j=1,\dots,n$, $\|A_{n,1}(j)\|_2^2\leq (\sum_{p\in\mathbb{Z}} |k|
|c_k(s_{\star})|)^2 \PE\{(\overline{X_n}-n\mu/2)^2\}=O(n)$ and
$n^{-1}(\sum_{j=1}^n \|A_{n,2}(j)\|_2^2)^{1/2}\leq C n^{1/2}
(\sum_{|k|\geq K_n} |k| |c_k(s_{\star})|)=o(\sqrt{n})$, using
\eqref{eq:sobolev}.
Using Minkowski's inequality, we obtain, for all $j=1,\dots,n$,
$$\|A_{n,3}(j)\|_2\leq 2\pi \sum_{|k|\leq K_n}\sum_{p\neq
  k} |k|
|c_p(s_{\star})|\left(\|X_j\varphi_{n,X}\{2\pi(p-k)f_{0}\}\|_2
+\|\dot{\varphi}_{n,X}\{2\pi(p-k)f_{0}\}\|_2\right)\; .
$$
Using that $|\varphi_{n,X}|\leq1$, Lemma~\ref{moments} which gives
$\PE[|\varphi_{n,X}\{2\pi(p-k)f_{0}\} |^2]=O(n^{-1})$ uniformly in
$p\neq k$, we obtain
$\|X_j\varphi_{n,X}\{2\pi(p-k)f_{0}\}\|_2\leq \|X_j-j\mu\|_2 +j\mu
n^{-1/2}=O(jn^{-1/2}+j^{1/2})$. By Lemma~\ref{moments},
$\|\dot{\varphi}_{n,X}\{2\pi(p-k)f_{0}\}\|_2=O(n^{1/2})$ uniformly in
$p\neq k$ leading thus to $n^{-1}(\sum_{j=1}^n \|A_{n,3}(j)\|_2^2)^{1/2}
=O(K_n^2)=o(\sqrt{n})$ by \eqref{eq:KnCond}. This concludes the proof
of \eqref{eq:XiPrime}.

We now prove~(\ref{ZetaPrime}). Using~(\ref{eq:zetaQ}) and~(\ref{eq:quadBoundcondX}) with $q=1$, we get
$$
\|\dot{\zeta}_n(f_{0})\|_2 \leq 2\pi n^{-2} \sum_{k=1}^{K_n}k
\|\varepsilon^T \Gamma_{q}(X,f)\varepsilon\|_2 =O\left(K_n^2\right)=o(\sqrt{n})
$$
by~(\ref{eq:KnCond}). Hence~(\ref{ZetaPrime}).

Let us now prove~(\ref{DePrime}). Using~(\ref{eq:DnDef}), we get
\begin{multline}\label{eq:DnPrime}
\dot{D}_n(f_{0})= \sum_{k=1}^{K_n}\sum_{p,q\in\zset}
c_p(s_{\star})\overline{c_q(s_{\star})}(-2\pi
k)\left\{
\dot{\varphi}_{n,X}[2\pi(p-k)f_{0}]\ \overline{\varphi_{n,X}[2\pi(q-k)f_{0}]}\right.\\
\left.+\varphi_{n,X}[2\pi(p-k)f_{0}]\ \overline{\dot{\varphi}_{n,X}[2\pi(q-k)f_{0}]}
\right\} \; .
\end{multline}
% \left\{
% \begin{array}{c}
% \dot{\varphi}_{n,X}[2\pi(p-k)f_{0}]\ \overline{\varphi_{n,X}[2\pi(q-k)f_{0}]}\\
% +\varphi_{n,X}[2\pi(p-k)f_{0}]\ \overline{\dot{\varphi}_{n,X}[2\pi(q-k)f_{0}]}
% \end{array}\right\}
% =\sum_{k=1}^{K_n}\sum_{(p,q)\neq (k,k)}
% c_p(s_{\star})\overline{c_q(s_{\star})}(-2\pi
% k)\left\{
% \begin{array}{c}
% \dot{\varphi}_{n,X}[2\pi(p-k)f_{0}]\ \overline{\varphi_{n,X}[2\pi(q-k)f_{0}]}\\
% +\varphi_{n,X}[2\pi(p-k)f_{0}]\ \overline{\dot{\varphi}_{n,X}[2\pi(q-k)f_{0}]}
% \end{array}\right\}.
Lemma~\ref{moments} gives that there exists a constant $C>0$, such that, for all $p\neq k$ and $q\neq k$,
%\begin{multline*}
$\PE(|\dot{\varphi}_{n,X}[2\pi(p-k)f_{0}]\ \overline{\varphi_{n,X}[2\pi(q-k)f_{0}]}|)
\leq \left\|\dot{\varphi}_{n,X}[2\pi(p-k)f_{0}]\right\|_2
\left\|\varphi_{n,X}[2\pi(q-k)f_{0}]\right\|_2
\leq C.$
%\end{multline*}
Using \eqref{eq:KnCond} and $\sum_p|c_p(s_{\star})|<\infty$, we get that the term $\sum_k\sum_{p\neq k,q\neq k}$ in the right-hand
side of~(\ref{eq:DnPrime}) is $o_p(\sqrt{n})$. Now, if $p=q=k$, the term in the curly brackets is equal to
zero. Hence~(\ref{eq:DnPrime}) can be rewritten as
$\dot{D}_n(f_{0})=\sum_{k=1}^{K_n} D_{n,k}+ o_p(\sqrt{n})$
where
\begin{multline*}
D_{n,k}=\sum_{q\in\zset}
c_k(s_{\star})\overline{c_q(s_{\star})}(-2\pi k)\left\{\dot{\varphi}_{n,X}(0)\ \overline{\varphi_{n,X}[2\pi(q-k)f_{0}]}
+\overline{\dot{\varphi}_{n,X}[2\pi(q-k)f_{0}]}\right\} \\
+ \sum_{p\in\zset}c_p(s_{\star})\overline{c_k(s_{\star})}(-2\pi k)
\left\{
\dot{\varphi}_{n,X}[2\pi(p-k)f_{0}]+
\varphi_{n,X}[2\pi(p-k)f_{0}]\overline{\dot{\varphi}_{n,X}(0)}\right\}\; .
\end{multline*}
%\begin{multline*}
%\dot{D}_n(f_{0})= \sum_{k=1}^{K_n}\left[\sum_{q\in\zset}
%c_k(s_{\star})\overline{c_q(s_{\star})}(-2\pi k)\left\{\dot{\varphi}_{n,X}(0)\ \overline{\varphi_{n,X}[2\pi(q-k)f_{0}]}
%+\overline{\dot{\varphi}_{n,X}[2\pi(q-k)f_{0}]}\right\} \right.\\
%\left.
%+ \sum_{p\in\zset}c_p(s_{\star})\overline{c_k(s_{\star})}(-2\pi k)
%\left\{
%\dot{\varphi}_{n,X}[2\pi(p-k)f_{0}]+
%\varphi_{n,X}[2\pi(p-k)f_{0}]\ \overline{\dot{\varphi}_{n,X}(0)}\right\}
%\right] + o_p(\sqrt{n}) \; .
%\end{multline*}
We will check that
 $\sum_{k>K_n}D_{n,k}=o_p(\sqrt{n})$.
Using the Fourier expansion of ${s_{\star}}$ and $\dot{s}_{\star}$, we obtain after some algebra,
%\begin{equation}\label{eq:Dnprim2}
$\dot{D}_n(f_{0})=(n f_{0})^{-1}\sum_{j=1}^n \left(X_j -\overline{X_n}\right) \dot{s}_{\star}(X_j)
s_{\star}(X_j)+o_{p}(\sqrt{n})$.
%\end{equation}
This yields~\eqref{DePrime} by Slutsky's Lemma. Indeed, $\mu/2-\sum_{l=1}^n X_l/n^2=o_p(1)$ and, by Proposition~\ref{Prop:CLT}, $n^{-1/2}\sum_{j=1}^n
(s_{\star} \dot{s}_{\star})(X_j)=O_p(1)$, thus we have
$
\left(\mu/2-n^{-2}\sum_{l=1}^n X_l\right)\sum_{j=1}^n\dot{s}_{\star}(X_j)s_{\star}(X_j)=o_p(\sqrt{n}).
$
To conclude the proof of~(\ref{lambprim}), we have to prove that
 $\sum_{k>K_n}D_{n,k}=o_p(\sqrt{n})$. By Minknowski
inequality,
$
\left\|\sum_{k>K_n} D_{n,k}\right\|_2\leq 2\pi\sum_{k>K_n} \sum_{q\neq k} |k|
|c_k(s_{\star})| |c_q(s_{\star})|\\(\|\dot{\varphi}_{n,X}(0)\overline{\varphi_{n,X}[2\pi(q-k)f_{0}]}\|_2
+\|\dot{\varphi}_{n,X}[2\pi(q-k)f_{0}]\|_2).
$
Using that $|\varphi_{n,X}|\leq 1$ and
$\|\varphi_{n,X}[2\pi(q-k)f_{0}]\|_2=O(n^{-1/2})$ by Lemma
\ref{moments} uniformly in $q\neq k$, we get $\|\dot{\varphi}_{n,X}(0)\overline{\varphi_{n,X}[2\pi(q-k)f_{0}]}\|_2
\leq
\|\overline{X_n}-(n+1)\mu/2\|_2+(n+1)\mu/2\|\varphi_{n,X}[2\pi(q-k)f_{0}]\|_2=O(\sqrt{n})$.
By \eqref{eq:sobolev} and Lemma \ref{moments}, we obtain
$
%\begin{equation}\label{eq:ramainderSnPrime}
\left\|\sum_{k>K_n} D_{n,k}\right\|_2=o(\sqrt{n})\; .
%\end{equation}
$
%By the Cauchy-Schwarz inequality, setting $C\eqdef \sum_{k>K_n}\sum_{q\in\zset} k^2 |c_k(s_{\star})|\,|c_q(s_{\star})|<\infty$,
%we have
%\begin{multline}
%\label{eq:ramainderSnPrime}
%\PE\left[\left|\sum_{k>K_n}\sum_{q\in\zset}  k c_k(s_{\star})
%\overline{c_q(s_{\star})}\left\{\dot{\varphi}_{n,X}(0)\ \overline{\varphi_{n,X}[2\pi(q-k)f_{0}]}
%+\overline{\dot{\varphi}_{n,X}[2\pi(q-k)f_{0}]}\right\}\right|^2\right]\\
%\leq C \,
%\sum_{k>K_n}\sum_{q\in\zset}|c_k(s_{\star})|\,|c_q(s_{\star})|  \PE\left[\left|\left\{...\right\}\right|^2\right] = o(n) \;,
%\end{multline}
%where we used that $\{\dots\}$ is zero for $q=k$ and, for $q\neq k$, by Lemma~\ref{moments},
%\begin{multline*}
%\PE\left[\left|\left\{...\right\}\right|^2\right]\leq
%2\PE\left[n^{-2}\left(\sum_{j=1}^n
%    X_j\right)^2 \left|\varphi_{n,X}[2\pi(q-k)f_{0}]\right|^2
%+ \left|\dot{\varphi}_{n,X}[2\pi(q-k)f_{0}]\right|^2\right] \\
%\leq
% 4 n^{-2}  \PE\left[\left(\sum_{j=1}^n (X_j-j\mu) \right)^2 \right] + 4 n^{-2} \left(\sum_{j=1}^n j\mu \right)^2\PE\left[
%     \left|\varphi_{n,X}[2\pi(q-k)f_{0}]\right|^2\right] + O(n) =O(n)  \; .
%\end{multline*}

\subsection{Proof of Eq.~\eqref{lambsec}}\label{sec:proof-eq.-eqrefllambsec}

Using that $\ddot{\Lambda}_n=\ddot{D}_n+\ddot{\xi}_n+\ddot{\zeta}_n$, applying Lemma~\ref{eta} with $q=2$ and
using~(\ref{eq:ConvPS}), the Relation~(\ref{lambsec}) is a consequence of the two following estimates, proved below,
\begin{align}
\label{eq:lambsec1}
&\ddot{D}_n(f_{0})=-n^2\mu^2(12\;f_{0})^{-1}\int_0^{1/f_0} \dot{s}_{\star}^2(t)dt \; (1+o_{p}(1)) \;,\\
\label{eq:lambsec2}
&\sup_{f:|f-f_{0}|\leq\rho_n/n}|\ddot{D}_n(f_{0})-\ddot{D}_n(f)|=o_p(n^2) \;,
\end{align}
for any decreasing sequence $(\rho_n)$ tending to zero.
In this proof section, we use the Fourier expansion~(\ref{fourierCoeff}) defined with $T=1/f_{0}$.
We now prove~(\ref{eq:lambsec1}). Using~(\ref{eq:DnDef}), we obtain
\begin{multline}\label{eq:DnSec}
\ddot{D}_n(f)=4\pi^2 \sum_{k=1}^{K_n}\sum_{p,q\in\zset} c_p(s_{\star})\overline{c_q(s_{\star})} k^2
\left\{
\ddot{\varphi}_{n,X}[2\pi(pf_{0}-kf)]\overline{\varphi_{n,X}[2\pi(qf_{0}-kf)]}\right.\\
\left.+2\dot{\varphi}_{n,X}[2\pi(pf_{0}-kf)]\overline{\dot{\varphi}_{n,X}[2\pi(qf_{0}-kf)]}
+\varphi_{n,X}[2\pi(pf_{0}-kf)]\overline{\ddot{\varphi}_{n,X}[2\pi(qf_{0}-kf)]}
\right\} \; .
\end{multline}
For $f=f_{0}$,
% since $\varphi_{n,X}(0)=1$, $\dot{\varphi}_{n,X}(0)=\rmi/n\sum_{j=1}^nX_j$ and
% $\ddot{\varphi}_{n,X}(0)=-1/n\sum_{j=1}^nX_j^2$,
we get
\begin{equation}\label{eq:DnSecFstar}
\frac1{n^{2}}\ddot{D}_n(f_{0})= \left(4\pi^2 \sum_{k=1}^{K_n} |c_k(s_{\star})|^2 k^2\right)
\left\{-\frac2{n^3}\sum_{j=1}^nX_j^2+\frac2{n^4}\left(\sum_{j=1}^nX_j\right)^2\right\}
+ G_n\;,
\end{equation}
\begin{multline*}
\text{where }
G_n=\frac{4\pi^2}{n^2} \sum_{k=1}^{K_n}\sum_{(p,q)\neq (k,k)} c_p(s_{\star})\overline{c_q(s_{\star})}k^2
\left\{
\ddot{\varphi}_{n,X}[2\pi(p-k)f_{0}]\overline{\varphi_{n,X}[2\pi(q-k)f_{0}]}\right.\\
\left.+2\dot{\varphi}_{n,X}[2\pi(p-k)f_{0}]\overline{\dot{\varphi}_{n,X}[2\pi(q-k)f_{0}]}
+\varphi_{n,X}[2\pi(p-k)f_{0}]\overline{\ddot{\varphi}_{n,X}[2\pi(q-k)f_{0}]}
\right\}.
\end{multline*}
As $n$ tends to infinity, the term between parentheses in~(\ref{eq:DnSecFstar}) tends to
$1/(2f_{0})\int_0^{1/f_{0}}{\dot{s}_{\star}}^2(t)\,dt$ and the term between curly brackets converges to $-2\mu^2/3+\mu^2/2$
in probability, and hence their product converges to the constant appearing in the right-hand side of~(\ref{eq:lambsec1}).
We conclude the proof of~(\ref{eq:lambsec1}) by showing that $G_n=o_p(1)$.
We split the summation $\sum_{p,q}$
in the definition of $G_n$ into three terms $\sum_{p\neq k,q\neq k}+\sum_{p=k,q\neq k}+\sum_{p\neq k,q=k}=:\sum_{i=l}^{3} G_{n,l}$.
Observe that, setting $C=2\pi\sum_p|c_p(s_{\star})|$,
$$
\PE[|G_{n,1}|]\leq C^2 \,K_n^3n^{-2}\,
\inf_{|t|>2\pi f_{0}}\{\PE[|\ddot{\varphi}_{n,X}(t)\varphi_{n,X}(t)|]+\PE[|\dot{\varphi}_{n,X}(t)|^2]\} \;.
$$
Using that $\PE[|\ddot{\varphi}_{n,X}(t)\varphi_{n,X}(t)|]^2\leq \PE[|\ddot{\varphi}_{n,X}(t)|^2]\PE[|\varphi_{n,X}(t)|^2]$, Lemma
\ref{moments} yields $G_{n,1}=o_p(1).$
Note that
\begin{multline*}
\PE[|G_{n,2}+G_{n,3}|]\leq \frac{8\pi^2}{n^2}\sum_{k=1}^{K_n} k^2 |c_k(s_{\star})|\sum_{q\neq k} |c_q|
\left\{\PE\left[|\ddot{\varphi}_{n,X}(0)|^2\right]^{1/2}
\PE\left[|\varphi_{n,X}[2\pi(q-k)f_{0}]|^2\right]^{1/2}\right.\\
\left. +2 \PE\left[|\dot{\varphi}_{n,X}(0)|^2\right]^{1/2}
\PE\left[|\dot{\varphi}_{n,X}[2\pi(q-k)f_{0}]|^2\right]^{1/2}
+\PE\left[|\ddot{\varphi}_{n,X}[2\pi(q-k)f_{0}]|^2\right]^{1/2}\right\}=O(n^{-1/2}),
\end{multline*}
by using that $\PE\left[n^{-2}\left(\sum_{j=1}^n X_j\right)^2\right]=O(n^2)$,
$\PE\left[n^{-2}\left(\sum_{j=1}^n X_j^2\right)^2\right]=O(n^4)$ and Lemma \ref{moments}.

We now prove~(\ref{eq:lambsec2}). In the expression of $\ddot{D}_n(f)$ given by  the right-hand side of~(\ref{eq:DnSec}), we
separate the summation $\sum_{p,q}$ into three terms $\sum_{p=k,q}+\sum_{p\neq k,q=k}+\sum_{p\neq k,q\neq k}$ denoted by
\begin{equation}\label{eq:DnSec3terms}
\ddot{D}_n(f) = \ddot{D}_{n,1}(f) +\ddot{D}_{n,2}(f) +\ddot{D}_{n,3}(f) \; .
\end{equation}
Using that $\sum_k|c_k(s_{\star})|\,|k|^3$ and $\sum_p|c_p(s_{\star})|$ are finite, and that
$\ddot{\varphi}_{n,X}\overline{\varphi_{n,X}}+\dot{\varphi}_{n,X}\overline{\dot{\varphi}_{n,X}}+\varphi_{n,X}\overline{\ddot{\varphi}_{n,X}}$
is Lipschitz  with Lipschitz constant at most $n^{-1}\sum_{j=1}^nX_j^3+n^{-2}\sum_{j=1}^nX_j\sum_{j=1}^nX_j^2=O_p(n^3)$,
one easily gets that
\begin{equation}\label{eq:DnSecFirst2}
\sup_{f:|f-f_{0}|\leq \rho_n/n}
\left|\ddot{D}_{n,1}(f) +\ddot{D}_{n,2}(f) - \ddot{D}_{n,1}(f_{0})  - \ddot{D}_{n,2}(f_{0}) \right| = O_p(\rho_nn^2) = o_p(n^2) \; .
\end{equation}
Let $(f_l)_{1\leq l\leq L_n}$ be a regular grid with mesh $\delta_n$ covering $[f_{0}-\rho_n/n,
f_{0}+\rho_n/n]$. Then,
\begin{multline}\label{DnSecDecomp}
\sup_{f:|f-f_{0}|\leq \rho_n/n} \left|\ddot{D}_{n,3}(f)-\ddot{D}_{n,3}(f_{0})\right|\\
\leq \sup_{l=1,\dots,L_n} \left|\ddot{D}_{n,3}(f_l)-\ddot{D}_{n,3}(f_{0})\right|
+\sup_{l=1,\dots,L_n} \sup_{f\in [f_l,f_{l+1}]} \left|\ddot{D}_{n,3}(f)-\ddot{D}_{n,3}(f_l)\right|.
\end{multline}
Using the same argument as above with $\sum_p|c_p(s_{\star})|<\infty$ and $\sum_{k=1}^{K_n}k^3=O(K_n^4)$, we get
that
$
\sup_{l=1,\dots,L_n} \sup_{f\in [f_l,f_{l+1}]} \left|\ddot{D}_{n,3}(f)-\ddot{D}_{n,3}(f_l)\right| = O_p\left(K_n^4\delta_nn^3\right) \; .
$
Since $K_n=o(n^{-1})$, there exists $N$ such that, for any $n\geq N$, any $f$ such that $|f-f_{0}|\leq1/n$ and any
$p\in\zset$ and $k=1,\dots,K_n$ such that $p\neq k$, we have $|pf_{0}-kf|\geq f_{0}/2$. Then proceeding as for
bounding $G_n$ above, we have, for any $n\geq N$ and any $f$ such that $|f-f_{0}|\leq1/n$,
$
\PE\left[\left|\ddot{D}_{n,3}(f)\right|\right] \leq C\,K_n^3\,n \; ,
$
where $C$ is some positive constant. From this, we obtain
$
\sup_{l=1,\dots,L_n} \left|\ddot{D}_{n,3}(f_l)-\ddot{D}_{n,3}(f_{0})\right| = O_p(L_nK_n^3\,n) \; ,
$
so that, for $\delta_n=n^{-3/2}$, implying $L_n=[\rho_n/(n\delta_n)]=o(n^{1/2})$,~(\ref{DnSecDecomp}) finally yields
$
\sup_{f:|f-f_{0}|\leq \rho_n/n} \left|\ddot{D}_{n,3}(f)-\ddot{D}_{n,3}(f_{0})\right| = O_p\left(K_n^4n^{3/2}\right),
$
which, with~(\ref{eq:DnSecFirst2}) and~(\ref{eq:DnSec3terms}), gives~(\ref{eq:lambsec2}).

\bibliographystyle{ims}
\bibliography{unequal}

\end{document}